\documentclass[onefignum,onetabnum]{siamart220329}

\usepackage{xparse}
\usepackage{booktabs}
\usepackage{tabu}
\usepackage{enumerate}
\usepackage{subfigure}

\newcommand{\Hs}{\mathrm H}
\newcommand{\tr}{\mathrm{tr}}
\newcommand{\dof}{\texttt{dof}}%

\newcommand{\nc}{\mathrm{nc}}
\newcommand{\cc}{\mathrm c}
\newcommand{\A}{\mathrm A}
\newcommand{\avg}[1]{\left\{\hspace{-2.1mm}\left\{ #1 \right\}\hspace{-2.1mm}\right\}}
\newcommand{\jp}[1]{\left[\!\!\left[ #1 \right]\!\!\right]}
\newcommand{\norm}[1]{|\hspace{-0.3mm}|\hspace{-0.3mm}|#1|\hspace{-0.3mm}|\hspace{-0.3mm}|}

\NewDocumentCommand{\dgal}{sO{}m}{%
  \IfBooleanTF{#1}
    {\dgalext{#3}}
    {\dgalx[#2]{#3}}%
}

\NewDocumentCommand{\dgalext}{m}{%
  \sbox0{%
    \mathsurround=0pt 
    $\left\{\vphantom{#1}\right.\kern-\nulldelimiterspace$%
  }%
  \sbox2{\{}%
  \ifdim\ht0=\ht2
    \{\kern-.45\wd2 \{#1\}\kern-.45\wd2 \}%
  \else
  \fi
}

\NewDocumentCommand{\dgalx}{om}{%
  \sbox0{\mathsurround=0pt$#1\{$}%
  \sbox2{\{}%
  \ifdim\ht0=\ht2
    \{\kern-.45\wd2 \{#2\}\kern-.45\wd2 \}%
  \else
    \mathopen{#1\{\kern-.5\wd0 #1\{}
    #2
    \mathclose{#1\}\kern-.5\wd0 #1\}}
  \fi
}

\usepackage{lipsum}
\usepackage{amsfonts}
\usepackage{algorithmic}
\usepackage{amsmath}
\usepackage{amssymb}
\usepackage{graphics}
\usepackage{graphicx}
\usepackage{footnote}
\usepackage{textcomp}
\usepackage{mathrsfs}
\usepackage{array}
\usepackage[maxfloats=99]{morefloats}
\usepackage{url}
\usepackage{cases}
\usepackage{mathscinet}
\usepackage[normalem]{ulem}
\usepackage{algorithm}
\usepackage{color}
\usepackage{cite}
\usepackage{bm}

\ifpdf
  \DeclareGraphicsExtensions{.eps,.pdf,.png,.jpg}
\else
  \DeclareGraphicsExtensions{.eps}
\fi

\newsiamremark{remark}{Remark}
\newsiamremark{hypothesis}{Hypothesis}
\crefname{hypothesis}{Hypothesis}{Hypotheses}
\newsiamthm{claim}{Claim}

\headers{VEM for HJB equations}{Ying Cai, Hailong  Guo, and Zhimin Zhang}

\title{Virtual element methods for HJB equations with Cordes coefficients}

\author{Ying Cai\thanks{Beijing Computational Science Research Center, Beijing 100193, China (\email{ycai@csrc.ac.cn}).}
\and Hailong Guo\thanks{School of Mathematics and Statistics, The University of Melbourne, Parkville, VIC 3010, Australia (\email{hailong.guo@unimelb.edu.au}).}
\and Zhimin Zhang\thanks{Department of Mathematics, Wayne State University, Detroit, MI 48202, USA (\email{ag7761@wayne.edu}).}
}

\usepackage{amsopn}

\usepackage{multirow}

\newtheorem{assumption}[theorem]{Assumption}
\graphicspath{{fig/}{fig/mesh/}{fig/adaptive}}

\begin{document}

\maketitle

\begin{abstract}
In this paper, we propose and analyze both conforming and nonconforming virtual element methods (VEMs) for the fully nonlinear second-order elliptic Hamilton–Jacobi–Bellman (HJB) equations with Cordes coefficients. By incorporating stabilization terms, we establish the well-posedness of the proposed methods, thus avoiding the need to construct a discrete Miranda–Talenti estimate. We derive the optimal error estimate in the discrete $H^2$ norm for both numerical formulations. Furthermore, a semismooth Newton’s method is employed to linearize the discrete problems. Several numerical experiments using the lowest-order VEMs are provided to demonstrate the efficacy of the proposed methods and to validate our theoretical results.
\end{abstract}



\section{Introduction}
\label{sec:introduction}
Let $\varOmega\subset\mathbb R^2$ be a bounded convex polygonal domain with boundary $\partial\varOmega$. We consider
the Hamilton-Jacobi-Bellman equation, which reads as follows
\begin{equation}\label{eq:HJB}
\sup_{\alpha\in\varLambda}(L^\alpha u-f^\alpha)=0\quad\text{  in } \varOmega,\qquad u=0\quad\text{on }\partial\varOmega,
\end{equation}
where $\varLambda$ is a compact metric space and
\[L^\alpha v:=\A^\alpha:D^2v+\bm b^\alpha\cdot\nabla v-c^\alpha v.\]
Here,  $\A, \bm b, c$ and $f$ belong to the space $C(\overline\varOmega,\varLambda)$. The collection functions $\{\A^\alpha\}_{\alpha\in\varLambda}$,  $\{\bm b^\alpha\}_{\alpha\in\varLambda}$, $\{c^\alpha\}_{\alpha\in\varLambda}$
and $\{f^\alpha\}_{\alpha\in\varLambda}$ are given as follows: for every index $\alpha\in\varLambda$,
\[\begin{array}{cc}
\A^\alpha: x\mapsto \A(x,\alpha),& \bm b^\alpha: x\mapsto \bm b(x,\alpha), \\
c^\alpha: x\mapsto c(x,\alpha), & f^\alpha: x\mapsto f(x,\alpha).
\end{array}
\]
A special case of problem \eqref{eq:HJB} is the linear elliptic equation in nondivergence form:
\begin{equation}\label{eq:Ndiv}
  Lu:=\A:D^2u+\bm b\cdot\nabla u-cu=f\quad\text{  in } \varOmega,\qquad u=0\quad\text{on }\partial\varOmega.
\end{equation}
In this case, we simply assume that $\A\in L^\infty(\varOmega)$ and $f\in\ L^2(\varOmega)$.

The study of the HJB equation \eqref{eq:HJB} originates from stochastic control problems, which are of significant importance in economics and finance. The primary challenge in developing numerical methods for \eqref{eq:HJB}, in addition to addressing its nonlinearity, is its non-divergence form. Linearizing \eqref{eq:HJB} results in a sequence of problems of the form \eqref{eq:Ndiv}. The well-posedness and the concept of the solution for the HJB equation depend on the regularity of the problem's parameters and the smoothness of the domain boundary, which has led to several distinct theoretical frameworks.

When the coefficient matrix $\A\in C(\overline\varOmega)$, X. Feng and M. Neilan et al. proposed non-standard primal finite element methods for solving \eqref{eq:Ndiv} in \cite{FHN2017,NM2017,FNS2018}.
The methods and the analysis of convergence are based on establishing discrete Calderon–Zygmund estimates. For the theory of $H^2$ solutions,  the analysis of problems \eqref{eq:Ndiv} and \eqref{eq:HJB}
relies on the following Miranda–Talenti estimate.
\begin{lemma}\label{lem:MT}
Suppose that $\varOmega$ is a bounded convex domain, then there holds
\begin{equation}\label{ineq:MT}
  \|D^2v\|_{L^2(\varOmega)}\leq \|\Delta v\|_{L^2(\varOmega)}\quad\forall v\in H^2(\varOmega)\cap H_0^1(\varOmega).
\end{equation}
\end{lemma}
In the $H^2$ theory, the coefficient matrix $\A$ is allowed to be discontinuous, as long as the Cordes condition is satisfied.


The numerical approximation of $H^2$ solutions was first studied by I. Smears and E. S\"uli. In their pioneering work, I. Smears and E. S\"uli \cite{MS2013} presented a discontinuous Galerkin (DG) method for equation \eqref{eq:Ndiv} with discontinuous coefficients and the Cordes condition. This technique was extended to the elliptic and parabolic HJB equations in \cite{SS2014,SS2016}, and to the nondivergence form equation on the curved domain in \cite{KE2019}. These methods bypass the discrete counterpart Miranda–Talenti inequality by adding auxiliary terms in the discrete formulation.
In \cite{NW2019}, a discrete Miranda–Talenti inequality was established, and the $C^0$ interior penalty method (C0IP) was analyzed for both equations \eqref{eq:Ndiv} and \eqref{eq:HJB}. In \cite{BK2021,KS2022-FCM}, adaptive C0IP methods and their convergence properties were developed. The unified analysis of DG and C0IP methods for HJB and Isaacs problems was presented in \cite{KS2021}. The DG and C0IP methods are applicable when the penalization parameters are chosen to be sufficiently large. A discrete Miranda-Talenti-type identity was introduced, and a new $C^0$ finite element method was proposed in \cite{Wu2021}; the distinctive feature of this formulation is that it does not involve a stabilization parameter. Moreover, an auxiliary space preconditioner for this method was constructed in \cite{GW2022}.

Recently, there has been increasing interest in developing numerical methods that utilize polygonal meshes other than triangular and quadrilateral shapes. This paper aims to investigate the discretization of the HJB equation \eqref{eq:HJB} using the virtual finite element method (VEM). Since the seminal paper by L. Beirão da Veiga et al. \cite{BBCMMR2013}, substantial advancements have been made in both research and practical applications of the VEM. We focus particularly on the $H^2$-type VEM. In \cite{BM2013}, F. Brezzi and L. D. Marini first constructed and analyzed a type of $H^2$-conforming VEM for plate bending problems. Subsequent discussions on this topic were provided in \cite{CM2016}, and similar ideas were extended to polyharmonic equations by P. F. Antonietti et al. in \cite{AMV2020,AMSm32021}. $C^1$-conforming virtual elements in three dimensions for linear elliptic fourth-order problems were considered in \cite{BDR2020}. In \cite{ZCZ2016}, a $C^0$-nonconforming VEM was proposed for plate bending problems. Additionally, in \cite{AMV2018,ZZCM2018}, Morley-type virtual elements were proposed for fourth-order problems. Furthermore, a robust nonconforming VEM for general fourth-order problems with varying coefficients was presented in \cite{DH2022}.

In a recent study, G. Bonnet et al. \cite{Bonnet2024} examined the $C^1$-conforming VEM solution for linear elliptic equations in nondivergence form with Cordes coefficients. Their analysis focused exclusively on the case where equation \eqref{eq:Ndiv} does not include low-order terms.
In this paper,  we design and analyze both  the $C^1$-conforming and $C^0$-nonconforming virtual element methods for the HJB equation \eqref{eq:HJB} with Cordes coefficients. The virtual element spaces are constructed based on \cite{BM2013} for $C^1$-conforming VEM and on \cite{ZCZ2016} for $C^0$-nonconforming VEM.
However, several challenges must be addressed when designing virtual element discretization.
First, it is well-known that one feature of VEM  is the computability of a specific projection of a function in the virtual element space, which is used in designing the discretization. However, it is unclear whether the projected function satisfies a discrete Miranda-Talenti inequality, complicating the numerical formulation. To address this challenge, inspired by \cite{SS2014}, we incorporate an appropriate stabilization term into the discrete formulation. This avoids the need to rely on the discrete Miranda-Talenti inequality to establish the existence and uniqueness of the solution to the discrete problem.  In contrast, \cite{Bonnet2024} does not include this auxiliary stabilization term and instead relies on assumptions about stability constants.
Second, since the problem we are considering involves low-order terms, calculating the $H^1$ and $L^2$ projections is necessary during the virtual element discretization. The spaces introduced in \cite{BM2013} and \cite{ZCZ2016} appear to be inadequate for these projections. To address this, we use the equivalent projection technique proposed in \cite{AABMR2013} to enlarge the original virtual element space and then factor out the high-order polynomial space. This approach ensures that the required virtual element space is obtained and that the corresponding projection operators are computable.
Third, given the fully nonlinear nature of problem \eqref{eq:HJB}, a linearization method must be considered. We employ a semismooth Newton’s method to solve the discretization problem.

The rest of our paper is organized as follows. In Section \ref{sec:review}, we review the theory of the strong solution to the HJB equation and recall the Cordes condition for this equation. Next, we propose an abstract virtual element framework for solving the HJB equation and derive an abstract error bound in Section \ref{sec:framework}. Section \ref{sec:conforming} investigates the $C^1$-conforming virtual element method for the HJB equation, including stability analysis and error estimation. In Section \ref{sec:nonconforming}, we present the $C^0$-nonconforming virtual element method, focusing on the estimation of the consistent error. Section \ref{sec:SemismoothNewton} introduces a semismooth Newton's method for linearizing the discretization problem. We report numerical experiments supporting our theoretical findings in Section \ref{sec:numerical}. Finally, conclusions are drawn in the last section.

\section{Review of the strong solution to the HJB equation}\label{sec:review}
For the HJB equation \eqref{eq:HJB}, we assume that the collection of coefficients is uniformly elliptic, meaning that there exist two positive constants $\varrho_1\leq\varrho_2$ such that
\begin{equation}\label{cond:elliptic}
\varrho_1|\bm\zeta|^2\leq \bm\zeta^t\A^\alpha(x)\bm\zeta\leq\varrho_2|\bm\zeta|^2\quad \forall\bm\zeta\in \mathbb R^2,\text{ a.e. in }\varOmega,\,\,\forall\alpha\in\varLambda,
\end{equation}
where  $|\bm\zeta|$ denotes the Euclidean norm of the vector $\bm\zeta$.
However, it is well-known that the uniformly elliptic condition is not sufficient for the well-posedness of the problem \eqref{eq:HJB}. Therefore, we assume that the following Cordes condition for the coefficients holds.
\begin{definition}\label{def:cordes} (Cordes condition for the HJB equation \eqref{eq:HJB}) The coefficients satisfy
\begin{itemize}
\item If $\bm b^\alpha\not\equiv\bm 0$ or $c^\alpha\not\equiv0$ for some $\alpha\in\varLambda$, there exist $\lambda>0$ and $\varepsilon\in(0,1]$ such that
\begin{equation}\label{eq:Cordes1}
\frac{|\A^\alpha|^2+|\bm b^\alpha|^2/(2\lambda)+(c^\alpha/\lambda)^2}{(\tr \A^\alpha+c^\alpha/\lambda)^2}\leq \frac{1}{2+\varepsilon}\quad \mathrm{ a.e.}\text{ in }\varOmega,\,\,\forall\alpha\in\varLambda,
\end{equation}
where $|\A^\alpha|$ denotes the Frobenius norm of the matrix $\A^\alpha$.
\item If $\bm b^\alpha\equiv\bm0$ and $c^\alpha=0$ for all $\alpha\in\varLambda$, then there exists an $\varepsilon\in (0,1]$ such that
\[\frac{|\A^\alpha|^2}{(\tr \A^\alpha)^2}\leq\frac1{1+\varepsilon}\quad \mathrm{ a.e.}\text{ in }\varOmega.\]
\end{itemize}
\end{definition}

For each $\alpha\in\varLambda$, we define
\[\gamma^\alpha:=
\begin{cases}
  \frac{\tr \A^\alpha+c^\alpha/\lambda}{|\A^\alpha|^2+|\bm b^\alpha|^2/(2\lambda)+(c^\alpha/\lambda)^2}, & \bm b^\alpha\not\equiv\bm 0 \text{ or } c^\alpha\not\equiv0\text{ for some } \alpha\in\varLambda, \\
  \frac{\tr \A^\alpha}{|\A^\alpha|^2}, & \bm b^\alpha\equiv\bm 0 \text{ and } c^\alpha\equiv0\text{ for all } \alpha\in\varLambda.
\end{cases}
\]
From the continuity of the given data, we know that $\gamma^\alpha\in C(\overline\varOmega,\varLambda)$. The fully nonlinear operator $F_\gamma:H^2(\varOmega)\to L^2(\varOmega)$ is  defined as
\[F_\gamma[v]:=\sup_{\alpha\in\varLambda}\gamma^\alpha(L^\alpha v-f^\alpha)\qquad \forall v\in H^2(\varOmega).\]

For $\lambda$ introduced in \eqref{eq:Cordes1}, we define a linear differential operator $L_\lambda$ by
\begin{equation}
    L_\lambda v:=\Delta v-\lambda v\qquad \forall v\in H^2(\varOmega).
\end{equation}
Let $H:=H^2(\varOmega)\cap H_0^1(\varOmega)$.  Define  the operator $\mathfrak A:H\to H^*$ by
\begin{equation}
\langle\mathfrak A(u),v\rangle:=\int_\varOmega F_\gamma[u]L_\lambda v,
\end{equation}
where $H^*$ is the dual space of $H$.
It can be shown that $\mathfrak A$ is strictly monotone, and Lipschitz continuous on $H$ (see \cite{SS2014}),   which means that there exists a unique function $u \in H$ such that
\begin{equation}
\langle\mathfrak A(u),v\rangle=0\qquad \forall v\in H.
\end{equation}
We conclude the following well-posedness result for the HJB equation \eqref{eq:HJB}, cf. \cite[Theorem 3]{SS2014}.
\begin{theorem}
  There exists a unique $u\in H$ that solves equation \eqref{eq:HJB} pointwise $\mathrm{a.e.}$ in $\varOmega$.
\end{theorem}

\section{The virtual element framework}\label{sec:framework}
In this section, we introduce a selection of essential general ingredients for discretizing the HJB problem using virtual element methods, followed by an abstract error estimate.

Let $\mathcal K_h$ be a decomposition of $\varOmega$ into non-overlapping convex polygons, and $h$ stand for the maximum of the diameters of elements in $\mathcal K_h$. Denote $\mathcal E_h$ as the set of edges in $\mathcal K_h$, and
$\mathcal E_h^\circ$ as the set of edges in the interior of the domain.
Moreover, for any $s>0$, we introduce the broken Sobolev space
\[H^s(\mathcal K_h):=\{v\in L^2(\varOmega):v|_K\in H^s(K)~~\forall K\in\mathcal K_h\},\]
equipped with the broken $H^s$ norm and semi-norm:
\[\|\cdot\|_{s,h}:=\left(\sum_{K\in\mathcal K_h}\|v\|_{s,K}^2\right)^{\frac12}\text{   and   }|\cdot|_{s,h}:=\left(\sum_{K\in\mathcal K_h}|v|_{s,K}^2\right)^{\frac12}\]
respectively. For a given measurable set $\omega$, we use $\mathbb P_k(\omega)$ to represent the polynomial space of degree  up to $k$ on $\omega$. In this paper, $\omega$ may be an edge or an element. Additionally, we define the
piecewise  polynomial space of degree up to $k$  over the mesh $\mathcal K_h$ as
\[\mathbb P_k(\mathcal K_h):=\{q\in L^2(\varOmega):q|_K\in\mathbb P_k(K)~~\forall K\in\mathcal K_h\}.\]

For any interior edge $e$, shared by two elements $K^l$ and $K^r$, the jump and the average of a piecewise smooth function $\varphi$ across $e$ are
defined by
\[\jp{\varphi}:=(\varphi|_{K^l})|_e-(\varphi|_{K^r})|_e,\qquad\avg{\varphi}:=\frac12\Big[(\varphi|_{K^l})|_e+(\varphi|_{K^r})|_e\Big].\]
 If $e\subset\partial\varOmega$, $\jp{\varphi}$ and $\avg{\varphi}$ are determined by the trace of $\varphi$ on $e$.
 For any $e\in\mathcal E_h$, we assign a unit tangent vector denoted by $\bm t_e$, and a unit normal vector denoted by $\bm n_e$. When there is no ambiguity, we
 will omit the subscript $e$ and simply denote them as $\bm t$ and $\bm n$.

For a given integer $k\geq2$, we introduce a subspace of $H^2(\mathcal K_h)$ with specific continuity as
\[
  H_k^{2,\nc}(\mathcal K_h):=\left\{v\in H_0^1(\varOmega)\cap H^2(\mathcal K_h):\int_e\jp{\frac{\partial v}{\partial\bm n}}q=0,
  ~\forall q\in\mathbb P_{k-2}(e),~\forall e\in\mathcal E_h^\circ\right\}.
\]

 As in  \cite{BBCMMR2013,AABMR2013,BM2013,BGS2017}, we
make the following assumption regarding the partition:
\begin{assumption}\label{ams:mesh}
  We denote $h_K$ and $h_e$ as the diameters of $K\in\mathcal K_h$ and $e\in\mathcal E_h$, respectively. We assume there exists a constant  $\rho>0$ such that

  \begin{enumerate}[(i)]
\item Every element $K\in\mathcal K_h$ is star-shaped with respect to every point of a disk $D$ with radius greater than $\rho h_K$;
\item For any $K\in\mathcal K_h$ and any $e\in\mathcal E_h$ with $e\subset\partial K$, it holds that $h_e\geq \rho h_K$;
\item For any $K\in\mathcal K_h$, there holds $h_K\geq\rho h$.
\end{enumerate}
\end{assumption}

On each element $K\in\mathcal K_h$, we define a nonlinear form and a bilinear form by:
\[A^K(w;v):=\langle\mathfrak A(w),v\rangle_K-(L_\lambda w,L_\lambda v)_K\]
\[B_*^K(w,v):=(D^2w,D^2v)_K+2\lambda(\nabla w,\nabla v)_K+\lambda^2(w,v)_K,\]
where $\langle\mathfrak A(w),v\rangle_K:=\int_KF_\gamma[w]L_\lambda v$.
The bilinear form $B_\theta^K(\cdot,\cdot)$ with $\theta\in[0,1]$ is given by
\[B_\theta^K(w,v):=\theta B_*^K(w,v)+(1-\theta)(L_\lambda w,L_\lambda v)_K.\]
The global forms are then defined as
\[A(w;v):=\sum_{K\in\mathcal K_h}A^K(w;v),\quad
B_\theta(w,v):=\sum_{K\in\mathcal K_h}B_\theta^K(w,v).\]
Finally, we introduce
\[a(w;v):=\sum_{K\in\mathcal K_h}a^K(w;v):=A(w;v)+B_{\frac12}(w,v).\]

It is straightforward to verify  that $|\cdot|_{2,h}$ is a norm
on the space $H_k^{2,\nc}(\mathcal K_h)$. Consequently, we define $\norm{\cdot}_{B,h}$-norm on $H_k^{2,\nc}(\mathcal K_h)$ by
\[\norm{\cdot}_{B,h}=\left(\sum_{K\in\mathcal K_h}\norm{\cdot}_{B,h,K}^2\right)^{\frac12}:=\left[\sum_{K\in\mathcal K_h}\Big(|\cdot|_{2,K}^2+2\lambda|\cdot|_{1,K}^2+\lambda^2\|\cdot\|_{0,K}^2\Big)\right]^{\frac12},\]
which is induced by the bilinear form $B_*(\cdot,\cdot):=\sum_{K\in\mathcal K_h}B_*^K(\cdot,\cdot)$ on $H_k^{2,\nc}(\mathcal K_h)$.

In the context of the virtual element method, we impose several assumptions on the discrete space, as well as on the discrete bilinear and nonlinear forms.
\begin{assumption}\label{ams:a}
For each fixed mesh size $h>0$ and an integer $k\geq2$, we assume the following conditions hold:
  \begin{enumerate}[(1)]
\item  The finite dimensional function space $V_{h,k}$ satisfies $V_{h,k}\subset H_k^{2,\nc}(\mathcal K_h)$;
\item  There exists a discrete nonlinear form $A_h:V_{h,k}\times V_{h,k}\to\mathbb R$ that  can be decomposed as
    \[A_h(u_h;v_h)=\sum_{K\in\mathcal K_h}A_h^K(u_h;v_h),\]
    where $A_h^K(\cdot;\cdot)$ is a discrete nonlinear form defined on $V_{h,k}^K\times V_{h,k}^K$ with $V_{h,k}^K$ being the restriction of
     $V_{h,k}$ to $K$. Moreover, for each element $K$, we have $\mathbb P_k(K)\subset V_{h,k}^K$. We also point out that the construction of
      $A_h^K(\cdot,\cdot)$ is associated with $f^\alpha$.
 \item There exists a discrete bilinear form $B_{h,\theta}:V_{h,k}\times V_{h,k}\to\mathbb R$ for $0\leq\theta\leq1$, which can be decomposed  as
        \[B_{h,\theta}(u_h,v_h)=\sum_{K\in\mathcal K_h}B_{h,\theta}^K(u_h,v_h),\]
        where $B_{h,\theta}^K(\cdot,\cdot)$ is a discrete bilinear form defined on $V_{h,k}^K\times V_{h,k}^K$, and satisfies the boundedness condition
        \begin{equation}\label{ineq:boundB}B_{h,\theta}^K(v_h,w_h)\leq C\norm{v_h}_{B,h,K}\norm{w_h}_{B,h,K}\quad\forall v_h,w_h\in V_{h,k}.\end{equation}
\item We define \[a_h(u_h;v_h):=\sum_{K\in\mathcal K_h}a^K(u_h;v_h):=
A_h(u_h;v_h)+B_{h,\frac12}(u_h,v_h).\]
Then there exist two positive constants $\beta_\star$ and $\beta^\star$, independent of $h$, such that
\begin{equation}\label{ineq:cov}
  \beta_\star\norm{u_h-v_h}_{B,h}^2\leq a_h(u_h;u_h-v_h)-a_h(v_h;u_h-v_h),
\end{equation}
and
\begin{equation}\label{ineq:Lip}
 |a_h(u_h;w_h)-a_h(v_h;w_h)|\leq \beta^\star\norm{u_h-v_h}_{B,h}\norm{w_h}_{B,h}
\end{equation}
for any $u_h, v_h, w_h\in V_{h,k}$.
\end{enumerate}
\end{assumption}

Actually, the final criterion in Assumption \ref{ams:a} implies that $a_h$ is Lipschitz continuous. By applying the Browder–Minty theorem, cf. \cite{RR2004}, together with \eqref{ineq:cov} and \eqref{ineq:Lip}, we can establish the following lemma regarding the existence and uniqueness of the solution to the discrete problem.
\begin{lemma}\label{lem:e&u}
There exists a unique solution $u_h \in V_{h,k}$ such that
\begin{equation}\label{eq:dispro}
  a_h(u_h;v_h)=0\qquad \forall v_h\in V_{h,k}.
\end{equation}
\end{lemma}

Finally, we end this section by establishing an abstract \textit{a priori} error bound for the formulation \eqref{eq:dispro}.
\begin{lemma}\label{lem:errorbound}
Let $u$ be the solution of the problem \eqref{eq:HJB} and let $u_h$ be the solution of the discrete problem \eqref{eq:dispro}.
If Assumptions \ref{ams:mesh} and \ref{ams:a} hold, then there exist a generic positive constant $C$, independent of $h$, such that
\begin{equation}\label{ineq:errorbound}
  \norm{u-u_h}_{B,h}\leq C\left[\inf_{v_h\in V_{k,h}}\Big(\varXi_1(v_h)+\varXi_2(v_h)\Big)+\inf_{p\in\mathbb P_k(\mathcal K_h)}\varXi_3(p)\right],
\end{equation}
where the functional $\varXi_1(v_h)$ is defined by
\[\varXi_1(v_h):=\norm{u-v_h}_{B,h},\]
which describes the approximation capability of the virtual finite element space. The consistency error term $\varXi_2(v_h)$ is given by
\[\varXi_2(v_h):=\sup_{w_h\in V_{h,k}\setminus\{0\}}\frac{|A_h(v_h;w_h)+B_{\frac12}(u,w_h)|}{\norm{w_h}_{B,h}},\]
and the polynomial consistency term takes the form
\[
  \varXi_3(p):=\norm{u-p}_{B,h}+\sup_{w_h\in V_{h,k}\setminus\{0\}}\frac1{\norm{w_h}_{B,h}}\sum_{K\in\mathcal K_h}|B_{\frac12}^K(p,w_h)-B_{h,\frac12}^K(p,w_h)|.
  \]
\end{lemma}

\begin{proof}
For any function $v_h\in V_{h,k}$, let  $w_h=u_h-v_h\in V_{h,k}$. Applying the inequality \eqref{ineq:cov},  we obtain
\begin{equation}\label{ineq:eb1}
\begin{split}
     \beta_\star\norm{u_h-v_h}_{B,h}^2&\leq a_h(u_h;w_h)-a_h(v_h;w_h)=-a_h(v_h;w_h)\\
     &=-\left(A_h(v_h;w_h)+B_{\frac12}(u,w_h)\right)-\left(B_{h,\frac12}(v_h,w_h)-B_{\frac12}(u,w_h)\right).
\end{split}
\end{equation}
For any $K\in\mathcal K_h$ and $p\in\mathbb P_k(\mathcal K_h)$, we have the decomposition
\begin{equation}\label{ineq:eb2}
\begin{split}
    &B_{h,\frac12}(v_h,w_h)-B_{\frac12}(u,w_h) \\
    =&\sum_{K\in\mathcal K_h}\Big(B_{h,\frac12}^K(v_h-p,w_h)-B_{\frac12}^K(u-p,w_h)\Big)+\sum_{K\in\mathcal K_h}\Big(B_{h,\frac12}^K(p,w_h)-B_{\frac12}^K(p,w_h)\Big).
\end{split}
\end{equation}
By the triangle inequality, the Cauchy-Schwarz inequality and \eqref{ineq:boundB}, we obtain
\begin{equation}\label{ineq:eb3}
  \sum_{K\in\mathcal K_h}\Big(B_{h,\frac12}^K(v_h-p,w_h)-B_{\frac12}^K(u-p,w_h)\Big)\leq C(\norm{u-v_h}_{B,h}+\norm{u-p}_{B,h})\norm{w_h}_{B,h}.
\end{equation}
The conclusion follows by using the triangle inequality and inequalities \eqref{ineq:eb1}-\eqref{ineq:eb3}.
\end{proof}

\section{The conforming Virtual element method}\label{sec:conforming}
In this section, we investigate the $C^1$-conforming virtual element method for solving the HJB equation and derive an optimal \textit{a priori} error estimate.

For an element or an edge $S$ in the mesh, and a given integer $m$, we denote $\mathscr M_m(S)$ as the set of scaled monomials
\[\mathscr M_m(S):=\left\{\left(\frac{\bm x-\bm x_S}{h_S}\right)^\beta:|\beta|\leq\ell\right\},\]
where $\bm x_S$ is the barycenter of $S$, $\beta=(\beta_1,\beta_2)$ is a nonnegative multi-index, $|\beta|=\beta_1+\beta_2$,  and $\bm x^\beta=x_1^{\beta_1}x_2^{\beta_2}$.
We also define the homogeneous scaled monomials set $\mathscr M_m^\circ(S)=\mathscr M_m(S)\setminus \mathscr M_{m-1}(S)$ with $\mathscr M_{-1}(S)=\emptyset$.

\subsection{Virtual element spaces and projection operators}
Let integer $k\geq 2$, we define $r=\max\{3,k\}$. For any element $K\in\mathcal K_h$,  following the ideas in \cite{LZWC2024,AABMR2013,CMS2017,BM2013,CM2016}, we  introduce the enhancement local $C^1$-conforming space:
\[\widetilde V_{h,k}^{K,\cc}=\left\{v\in H^2(K):\Delta^2v|_K\in \mathbb P_k(K), v|_e\in \mathbb P_r(e), \frac{\partial v}{\partial \bm n}\Big|_e\in\mathbb P_{k-1}(e)\,\forall e\in \partial K\cap\mathcal E_h\right\}.\]
The associated degrees of freedom are given by:
\begin{itemize}
  \item The value of $v(\xi)$, for any vertex $\xi$ in $K$;
  \item The values of $h_\xi\frac{\partial v(\xi)}{\partial x_1}$ and $h_\xi\frac{\partial v(\xi)}{\partial x_2}$, for any vertex $\xi$ in $K$;
  \item For $r\geq4$, the moments $\frac1{h_e}\int_eqv$, for any $q\in\mathscr M_{r-4}(e)$ and $e\subset\partial K$;
  \item For $k\geq 3$, the moments $\int_eq\frac{\partial v}{\partial\bm n}$, for any $q\in\mathscr M_{k-3}(e)$ and $e\subset\partial K$;
  \item For $k\geq 2$, the moments $\frac1{h_K^2}\int_Kqv$, for any $q\in \mathscr M_{k-2}(K)$.
\end{itemize}
Here, for each vertex $\xi$, $h_\xi$ represents the characteristic length of $\xi$; for example, $h_\xi$ can be taken as the average of the diameters of the elements sharing the $\xi$ as a vertex.
By the standard argument of the virtual element (cf. \cite{AABMR2013,BM2013}), it can be shown that the local space $\widetilde V_{h,k}^{K,\cc}$ is unisolvent.

For any function $\psi\in\widetilde V_{h,k}^{K,\cc}$ and $q\in\mathbb P_k(K)$, one can exactly compute $(D^2\phi,D^2q)_K$ by utilizing the degrees of freedom, as shown in  \cite{BM2013}.  Consequently, we can construct a projection operator
$\Pi_{k,\Hs}^{K,\cc}:\widetilde V_{h,k}^{K,\cc}\to\mathbb P_k(K)\subset \widetilde V_{h,k}^{K,\cc}$ as follows, for any $\psi\in \widetilde V_{h,k}^{K,\cc}$,
\[
\left\{
    \begin{array}{l}
       (D^2\Pi_{k,\Hs}^{K,\cc}\psi,D^2q)_K=(D^2\psi,D^2q)_K\qquad \forall q\in\mathbb P_k(K)\\
       \\
       \widehat{\Pi_{k,\Hs}^{K,\cc}\psi} = \widehat \psi,\quad\widehat{\nabla\Pi_{k,\Hs}^{K,\cc}\psi} = \widehat{\nabla \psi},
    \end{array}
\right.
\]
where the \textit{quasi-average} $\widehat\psi$ is determined by
\[\widehat\psi:=\frac1\ell\sum_{j=1}^\ell\psi(\xi^j),\] and $\xi^j$, $j=1,\cdots,\ell$, are vertices of $K$.

Now, we are in a position to restrict $\widetilde V_{h,k}^{K,\cc}$ to a subspace $V_{h,k}^{K,\cc}$. Specifically, we define
\[V_{h,k}^{K,\cc}=\left\{v\in \widetilde V_{h,k}^{K,\cc}:(\Pi_{k,\Hs}^{K,\cc} v-v,q^\circ)_K=0\quad\forall q^\circ\in\bigcup_{j=k-3}^k\mathscr M_j^\circ(K)\right\}.\]
A function in $V_{h,k}^{K,\cc}$ can be uniquely determined by the following degrees of freedom (cf. \cite{AABMR2013,LZWC2024}):
\begin{itemize}
  \item The value of $v(\xi)$, for any vertex $\xi$ in $K$;
  \item The values of $h_\xi\frac{\partial v(\xi)}{\partial x_1}$ and $h_\xi\frac{\partial v(\xi)}{\partial x_2}$,   for any vertex $\xi$ in $K$;
  \item For $r\geq 4$, the moments $\frac1{h_e}\int_eqv$, for any $q\in\mathscr M_{r-4}(e)$ and $e\subset\partial K$;
  \item For $k\geq 3$, the moments $\int_eq\frac{\partial v}{\partial\bm n}$, for any $q\in\mathscr M_{k-3}(e)$ and $e\subset\partial K$;
  \item For $k\geq 4$, the moments $\frac1{h_K^2}\int_Kqv$, for any $q\in \mathscr M_{k-4}(K)$.
\end{itemize}
In Figure \ref{fig:dofc},  we illustrate the above degrees of freedom in the case of $k=2$ and $k=3$ for the sake of intuition.
\begin{figure}
   \centering
  \includegraphics[width=0.8\textwidth]{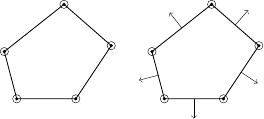}
  \caption{The degrees of freedom with $k=2$ (left) and $k=3$ (right)}
   \label{fig:dofc}
\end{figure}

The global virtual element space $V_{h,k}^\cc$ now is given by
\[V_{h,k}^\cc:=\left\{v\in H:v|_K\in V_{h,k}^{K,\cc}~~\forall K\in\mathcal K_h\right\}.\]

We now focus on the construction of the discrete nonlinear form and bilinear form for the conforming virtual element method. To this end, for an element $K \in \mathcal{K}_h$ and an integer $j$, we define the $L^2$ projection operator $\Pi_j^{K,\cc}$ onto the space $\mathbb{P}_j(K)$. Given a function $v$, the projection $\Pi_j^{K,\cc}v \in \mathbb{P}_j(K)$ satisfies
\[(\Pi_j^{K,\cc}v,q)_K=(v,q)_K\quad\forall q\in\mathbb P_j(K).\]
The definition of $\Pi_j^{K,\cc}$ can be generalized to vector or matrix-valued functions directly. We would like to comment that for any function $v\in V_{h,k}^{K,\cc}$, $\Pi_k^{K,\cc}v$, $\Pi_{k-1}^{K,\cc}\nabla v$, and
$\Pi_{k-2}^{K,\cc}D^2v$ are all computable through the degrees of freedom. In fact, $\Pi_k^{K,\cc} v$ is computable, which is evident from the definition of $V_{h,k}^{K,\cc}$ and the moments $\frac{1}{h_K^2}\int_K qv$, for any $q \in \mathscr{M}_{k-4}(K)$. In particular, if $k \leq 4$, we have $\Pi_k^{K,\cc}v = \Pi_{k,\Hs}^{K,\cc}v$.

On the other hand, for any $\Psi\in[\mathbb P_{k-2}(K)]^{2\times2}$ and $\bm w\in [\mathbb P_{k-1}(K)]^2$, using integration by parts twice yields:
\begin{equation}\label{eq:cp1}
  (D^2v,\Psi)_K=(\frac{\partial v}{\partial\bm n},\bm n^T\Psi\bm n)_{\partial K}+(\frac{\partial v}{\partial\bm t},\bm t^T\Psi\bm n)_{\partial K}-(v,\mathrm{div}\Psi\bm n)_{\partial K}+(v,\mathrm{div}\,\mathrm{div}\Psi)_K,
\end{equation}
and
\begin{equation}\label{eq:cp2}
  (\nabla v,\bm w)_K=(v,\bm w\cdot\bm n)_{\partial K}-(v,\mathrm{div}\bm w)_K=(v,\bm w\cdot\bm n)_{\partial K}-(\Pi_k^{K,\cc}v,\mathrm{div}\bm w)_K,
\end{equation}
which clarify that $\Pi_{k-2}^{K,\cc}D^2v$ and $\Pi_{k-1}^{K,\cc}\nabla v$ are computable though the degrees of freedom and $\Pi_k^{K,\cc}v$.

First, we define an approximation of $L^\alpha$ as:
\[\widehat L^{\alpha,\cc} v:=\A^\alpha:\Pi_{k-2}^{K,\cc}D^2v+\bm b^\alpha\cdot\Pi_{k-2}^{K,\cc}\nabla v-c^\alpha \Pi_{k-2}^{K,\cc}v.\]
Then we define
\[\widehat F_\gamma^\cc[v]:=\sup_{\alpha\in\varLambda}\gamma^\alpha(\widehat L^{\alpha,\cc} v-f^\alpha).\]
Similarly, $\widehat L_\lambda^\cc$ takes the form
\[\widehat L_\lambda^\cc v:=\tr\Pi_{k-2}^{K,\cc}D^2v-\lambda\Pi_{k-2}^{K,\cc}v.\]
We set
\[A_h^{K,\cc}(w_h;v_h):=(\widehat F_\gamma^\cc[w_h],\widehat L_\lambda^\cc v_h)_K-(\widehat L_\lambda^\cc w_h,\widehat L_\lambda^\cc v_h)_K+S^{K,\cc}(w_h-\Pi_{k,\Hs}^{K,\cc}w_h,v_h-\Pi_{k,\Hs}^{K,\cc}v_h),\]
where $S^{K,\cc}(\cdot,\cdot)$ is any symmetric and positive define bilinear form such that
\begin{equation}\label{ineq:Ss}
  c_\star\norm{v_h-\Pi_{k,\Hs}^{K,\cc}v_h}_{B,h,K}^2\leq S^{K,\cc}(v_h-\Pi_{k,\Hs}^{K,\cc}v_h,v_h-\Pi_{k,\Hs}^{K,\cc}v_h) \leq c^\star\norm{v_h-\Pi_{k,\Hs}^{K,\cc}v_h}_{B,h,K}^2,
\end{equation}
for any $v_h\in V_h^{K,\cc}$, and positive constants $c_\star$, $c^\star$ independent of $h$ and $K$. A typical choice for $S^{K,\cc}$ is based on the scalar product of the local degrees of freedom scaled with $\lambda$. More precisely,
we take
\[S^{K,\cc}(w_h,v_h):=(h_K^{-2}+2\lambda+\lambda^2h_K^2)\sum_{j=1}^{n_K^\cc}\dof_j^\cc(w_h)\cdot\dof_j^\cc(v_h).\]
Here $n_K^\cc$ represents the number of degrees of the freedom on element $K$, and $\dof_j^\cc$ denotes the $j$th degrees of freedom.

Next we define the bilinear forms $B_{h,*}^{K,\cc}(\cdot,\cdot)$ and $B_{h,\theta}^{K,\cc}(\cdot,\cdot)$ by
\begin{equation*}\begin{split}
B_{h,*}^{K,\cc}(w_h,v_h):=(\Pi_{k-2}^{K,\cc}D^2w_h,\Pi_{k-2}^{K,\cc}D^2v_h)_K+&2\lambda(\Pi_{k-2}^{K,\cc}\nabla w_h,\Pi_{k-2}^{K,\cc}\nabla v_h)_K\\
&+\lambda^2(\Pi_{k-2}^{K,\cc}w_h,\Pi_{k-2}^{K,\cc}v_h)_K,
\end{split}\end{equation*}
and
\[B_{h,\theta}^{K,\cc}(w_h,v_h)=\theta B_{h,*}^{K,\cc}(w_h,v_h)+(1-\theta)(\widehat L_\lambda^\cc w_h,\widehat L_\lambda^\cc v_h)_K,\quad \theta\in[0,1].\]

Finally, we define
\[a_h^{K,\cc}(w_h;v_h):=A_h^{K,\cc}(w_h;v_h)+B_{h,\frac12}^{K,\cc}(w_h,v_h).\]

\begin{remark}
The stabilization term $B_{h,*}^{K,\cc}(\cdot,\cdot)$ can be represented in some equivalent ways \cite[Lemma 5.1]{KS2021}.
\end{remark}



The $C^1$-conforming virtual element scheme for solving HJB equation is now formulated as: Find $u_h^\cc\in V_{h,k}^\cc$ such that:
\begin{equation}\label{eq:c1vem}
  a_h^\cc(u_h^\cc;v_h)=0\qquad\forall v_h\in V_{h,k}^\cc.
\end{equation}

\subsection{Stability analysis and error estimate}
In this subsection, we aim to discuss the well-posedness of our numerical scheme \eqref{eq:c1vem} and provide a detailed optimal \textit{a priori} estimate. To begin, we first present the following lemma associated with the Cordes condition.
\begin{lemma}\label{lem:cordes}
For any open set $U$, and $w,v\in H^2(U)$, if the Cordes condition in Definition \ref{def:cordes} holds, then the following inequality is true:
\begin{equation}\label{ineq:co}
  |\widehat F_\gamma^\cc[w]-\widehat F_\gamma^\cc[v]-\widehat L_\lambda^\cc(w-v)|\leq\sqrt{1-\varepsilon}\sqrt{|\Pi_{k-2}^{K,\cc}D^2z|^2+2\lambda|\Pi_{k-2}^{K,\cc}\nabla z|^2+\lambda^2|\Pi_{k-2}^{K,\cc}z|^2},
\end{equation}
with $z=w-v$.
\end{lemma}
\begin{proof}
  The proof included below mainly follows \cite{SS2014,Wu2021}, with specific details provided for completeness. We start by applying the Cauchy-Schwarz inequality to obtain:
  \begin{equation}\label{ineq:co1}
    \begin{split}
       &|\widehat F_\gamma^\cc[w]-\widehat F_\gamma^\cc[v]-\widehat L_\lambda^\cc(w-v)| \\
       =~&|\sup_{\alpha\in\varLambda}[\gamma^\alpha(\widehat L^{\alpha,\cc}w-f^\alpha)-\widehat L_\lambda^\cc w]-\sup_{\alpha\in\varLambda}[\gamma^\alpha(\widehat L^{\alpha,\cc}v-f^\alpha)-\widehat L_\lambda^\cc v]| \\
    \leq~&\sup_{\alpha\in\varLambda}|\gamma^\alpha\widehat L^{\alpha,\cc}z-\widehat L_\lambda^\cc z| \\
    \leq~&\sup_{\alpha\in\varLambda}\Big(|\gamma^\alpha\A^\alpha-I||\Pi_{k-2}^{K,\cc}D^2z|+|\gamma^\alpha||\bm b^\alpha||\Pi_{k-2}^{K,\cc}\nabla z|+|\lambda-\gamma^\alpha c^\alpha||\Pi_{k-2}^{K,\cc}z|\Big) \\
    \leq~&\left(\sup_{\alpha\in\varLambda}\sqrt{C^\alpha}\right)\sqrt{|\Pi_{k-2}^{K,\cc}D^2z|^2+2\lambda|\Pi_{k-2}^{K,\cc}\nabla z|^2+\lambda^2|\Pi_{k-2}^{K,\cc}z|^2},
    \end{split}
  \end{equation}
  where $C^\alpha$ has the expression:
  \[C^\alpha:=|\gamma^\alpha\A^\alpha-I|^2+\frac{|\gamma^\alpha|^2|\bm b^\alpha|^2}{2\lambda}+\frac{|\lambda-\gamma^\alpha c^\alpha|^2}{\lambda^2}.\]
  According to the Cordes condition and the definition of $\gamma^\alpha$, we have $C^\alpha \leq 1 - \varepsilon$. Thus, the proof is completed.
\end{proof}

To simplify the notation, we define
\[|\cdot|_{B,h,\Pi,K}:=\Big(|\Pi_{k-2}^{K,\cc}D^2\cdot|_{0,K}^2+2\lambda|\Pi_{k-2}^{K,\cc}\nabla \cdot|_{0,K}^2+\lambda^2|\Pi_{k-2}^{K,\cc}\cdot|_{0,K}^2\Big)^{\frac12}\]and
\[|\cdot|_{B,h,\Pi}:=\left(\sum_{K\in\mathcal K_h}|\cdot|_{B,h,\Pi,K}^2\right)^{\frac12}.\]

It is obvious that \eqref{ineq:boundB} holds for $B_{h,\theta}^{K,\cc}$,
 and we shall verify \eqref{ineq:cov} and \eqref{ineq:Lip}.
\begin{lemma}\label{lem:e&uc}
Inequalities \eqref{ineq:cov} and \eqref{ineq:Lip} hold for $a_h^\cc$, and then the problem \eqref{eq:c1vem} admits a unique solution.
\end{lemma}
\begin{proof}
For any $u_h, v_h\in V_{h,k}^\cc$, we obtain by the definition of $a_h^\cc(\cdot,\cdot)$ that
\begin{equation}\label{ineq:e&uc1}
\begin{split}
  & a_h^\cc(u_h;u_h-v_h)-a_h^\cc(v_h;u_h-v_h) \\
 =~& \sum_{K\in\mathcal K_h}\Big(A_h^{K,\cc}(u_h;u_h-v_h)-A_h^{K,\cc}(v_h;u_h-v_h)\Big)
 + \sum_{K\in\mathcal K_h}B_{h,\frac12}^{K,\cc}(u_h-v_h,u_h-v_h)\\
 :=~&J_1+J_2
\end{split}
\end{equation}
To estimate $J_1$, according to the definition of $A_h^{K,\cc}(\cdot,\cdot)$, we first consider
\[(\widehat F_\gamma^\cc[u_h]-\widehat L_\lambda^\cc u_h,\widehat L_\lambda^\cc(u_h-v_h))_K-(\widehat F_\gamma^\cc[v_h]-\widehat L_\lambda^\cc v_h,\widehat L_\lambda^\cc(u_h-v_h))_K=:\varTheta_K^\cc.\]
We estimate $\varTheta_K^\cc$ by using Lemma \ref{lem:cordes} and the Cauchy-Schwarz inequality:
\begin{equation}\label{ineq:e&uc2}
\begin{split}
  \varTheta_K^\cc&=(\widehat F_\gamma^\cc[u_h]-\widehat F_\gamma^\cc[v_h]-\widehat L_\lambda^\cc(u_h-v_h),\widehat L_\lambda^\cc(u_h-v_h))_K\\
  &\geq -\sqrt{1-\varepsilon}|u_h-v_h|_{B,h,\Pi,K}\|\widehat L_\lambda^\cc(u_h-v_h)\|_{0,K}.
  \end{split}
\end{equation}
Summing all elements and applying Young's inequality,  we have
\begin{equation}\label{ineq:e&uc3}
 \sum_{K\in\mathcal K_h}\varTheta_K^\cc\geq-\frac{1-\varepsilon}2|u_h-v_h|_{B,h,\Pi}^2-\frac12\sum_{K\in\mathcal K_h}\|\widehat L_\lambda^\cc(u_h-v_h)\|_{0,K}^2.
\end{equation}
On the other hand, by \eqref{ineq:Ss}, we get
\begin{equation}\label{ineq:e&uc4}
\begin{split}
&S^{K,\cc}(u_h-v_h-\Pi_{k,\Hs}^{K,\cc}(u_h-v_h),u_h-v_h-\Pi_{k,\Hs}^{K,\cc}(u_h-v_h))\\ \geq~& c_\star\norm{u_h-v_h-\Pi_{k,\Hs}^{K,\cc}(u_h-v_h)}_{B,h,K}^2.
\end{split}
\end{equation}
Consequently,  considering \eqref{ineq:e&uc3}, \eqref{ineq:e&uc4}, \eqref{ineq:Ss} and the definition of $J_2$, we obtain   
\[J_1+J_2\geq\frac\varepsilon2|u_h-v_h|_{B,h,\Pi}^2+c_\star\norm{u_h-v_h-\Pi_{k,\Hs}^{K,\cc}(u_h-v_h)}_{B,h}^2\geq C\norm{u_h-v_h}_{B,h}^2.\]

Next, we consider the Lipschitz continuity of $a_h^\cc$. For any $u_h, v_h, w_h\in V_{h,k}^\cc$, we find that
\begin{equation}\label{ineq:e&uc5}
\begin{split}
&a_h^\cc(u_h;w_h)-a_h^\cc(v_h;w_h)\\
=~&\sum_{K\in\mathcal K_h}(\widehat F_\gamma^\cc[u_h]-\widehat F_\gamma^\cc[v_h]-\widehat L_\lambda^\cc(u_h-v_h),\widehat L_\lambda^\cc w_h)_K\\
&+\sum_{K\in\mathcal K_h}S^{K,\cc}(u_h-\Pi_{k,\Hs}^{K,\cc}u_h-(v_h-\Pi_{k,\Hs}^{K,\cc}v_h),w_h-\Pi_{k,\Hs}^{K,\cc}w_h)+B_{h,\frac12}^\cc(u_h-v_h,w_h)\\
:=~&R_1+R_2+R_3.
\end{split}
\end{equation}
We derive by Lemma \ref{lem:cordes} and the Cauchy-Schwarz inequality that
\begin{equation}\label{ineq:e&uc6}
|R_1|\leq C\norm{u_h-v_h}_{B,h}\norm{w_h}_{B,h}.
\end{equation}
Using \eqref{ineq:Ss}, the Cauchy-Schwarz inequality and the definition of $\Pi_{k,\Hs}^{K,\cc}$, we have
\begin{equation}\label{ineq:e&uc7}
|R_2|\leq C\norm{u_h-v_h}_{B,h}\norm{w_h}_{B,h}.
\end{equation}
For $R_3$, applying the Cauchy-Schwarz inequality leads to
\begin{equation}\label{ineq:e&uc8}
|R_3|\leq C\norm{u_h-v_h}_{B,h}\norm{w_h}_{B,h}.
\end{equation}
Combing \eqref{ineq:e&uc6}-\eqref{ineq:e&uc8} with \eqref{ineq:e&uc5}, we conclude that $a_h^\cc$ is Lipschitz continuous.  Thus, by Lemma \ref{lem:e&u}, the problem \eqref{eq:c1vem} admits a unique solution.
\end{proof}

To establish an upper bound estimate of the virtual element error,
we define the canonical interpolation  $v_I^\cc|_K\in V_{h,k}^{K,\cc}$ for a smooth function $v$ and any element $K \in \mathcal{K}_h$ by
\[\dof_j^\cc(v-v_I^\cc)=0,\quad j=1,\cdots, n_K^\cc.\]
Under the mesh assumptions in Assumption \ref{ams:mesh}, $v_I^\cc$ satisfies the following approximation property, with details found in \cite{BBCMMR2013,BM2013,AABMR2013,LZWC2024} and the references therein:
 \begin{lemma}\label{lem:appc}
For $m=0,1,2$, $2< s\leq k+1$ and $v\in H^s(K)$, there holds that
\begin{equation}\label{ineq:appc}
  \|v-v_I^\cc\|_{m,K}\leq Ch^{s-m}\|v\|_{s,K}.
\end{equation}
\end{lemma}

From the abstract error estimate established in Lemma \ref{lem:errorbound}, we obtain
\begin{equation}\label{ineq:errC}
  \norm{u-u_h^\cc}_{B,h}\leq C\left[\inf_{v_h\in V_{k,h}}\Big(\varXi_1^\cc(v_h)+\varXi_2^\cc(v_h)\Big)+\inf_{p\in\mathbb P_k(\mathcal K_h)}\varXi_3^\cc(p)\right],
\end{equation}
with
\[\varXi_1^\cc(v_h):=\norm{u-v_h}_{B,h},\]
the consistency error term $\varXi_2(v_h)$ is given by
\[\varXi_2^\cc(v_h):=\sup_{w_h\in V_{h,k}^\cc\setminus\{0\}}\frac{|A_h^\cc(v_h;w_h)+B_{\frac12}(u,w_h)|}{\norm{w_h}_{B,h}},\]
and the polynomial consistency term
\[
  \varXi_3^\cc(p):=\norm{u-p}_{B,h}+\sup_{w_h\in V_{h,k}^\cc\setminus\{0\}}\frac1{\norm{w_h}_{B,h}}\sum_{K\in\mathcal K_h}|B_{\frac12}^K(p,w_h)-B_{h,\frac12}^{K,\cc}(p,w_h)|.
  \]
Next we will estimate $\varXi_j^\cc$, $j=1,2,3$, respectively, by taking $v_h=u_I^\cc$. For $\varXi_1^\cc(u_I^\cc)$, we have the following bound.
\begin{lemma}\label{lem:errC1}
If the true solution $u\in H_0^1(\varOmega)\cap H^s(\varOmega)$, $s>2$, then we have
\begin{equation}\label{ineq:errC1}
  \varXi_1^\cc(u_I^\cc)\leq Ch^{\sigma_{k,s}}\|u\|_{s,\varOmega},
\end{equation}
where here and in what follows, $\sigma_{k,s}:=\min\{k-1,s-2\}$.
\end{lemma}
\begin{proof}
  The result is a direct consequence of Lemma \ref{lem:appc} and the definition of $\norm{\cdot}_{B,h}$.
\end{proof}

We now estimate $\varXi_2^\cc(u_I^\cc)$ and we first establish the following consistency Lemma.
\begin{lemma}\label{lem:cst}
Suppose $w\in H_0^1(\varOmega)\cap H^s(\varOmega)$ with $s>\frac52$. Then for any $v\in  H_k^{2,\nc}(\mathcal K_h)$ we have
\begin{equation}\label{oneq:cst}
  B_*(w,v)=\sum_{K\in\mathcal K_h}(L_\lambda w,L_\lambda v)_K+\mathscr N(w,v)
\end{equation}
where
\[\mathscr N(w,v):=\sum_{e\in\mathcal E_h^\circ}(\frac{\partial^2w}{\partial\bm t^2}+\lambda w,\jp{\frac{\partial v}{\partial\bm n}})_e.\]
In particular, we have $\mathscr N(w,v)=0$ for any $v\in V_{h,k}^\cc$.
\end{lemma}
\begin{proof}
For any $K\in\mathcal K_h$, using integration by parts one can obtain (cf. \cite{MS2013})
 \begin{equation}\label{oneq:cst1}
 \begin{split}
  & \sum_{K\in\mathcal K_h}(D^2w,D^2v)_K-\sum_{K\in\mathcal K_h}(\Delta w,\Delta v)_K   \\
     =~~ &\sum_{e\in\mathcal E_h^\circ}(\frac{\partial^2w}{\partial\bm t^2},\jp{\frac{\partial v}{\partial\bm n}})_e -
     \sum_{e\in\mathcal E_h}(\frac{\partial^2w}{\partial\bm t\partial\bm n},\jp{\frac{\partial v}{\partial\bm t}})_e.
 \end{split}
\end{equation}
Using  the fact that $\frac{\partial v}{\partial\bm t}$ is continuous on $e$, $\forall e\in\mathcal E_h$, we claim that the second term in the
right hand of \eqref{oneq:cst1} is zero and we obtain that
 \begin{equation}\label{oneq:cst3}
\sum_{K\in\mathcal K_h}\Big[(D^2w,D^2v)_K-(\Delta w,\Delta v)_K\Big]=
\sum_{e\in\mathcal E_h^\circ}(\frac{\partial^2w}{\partial\bm t^2},\jp{\frac{\partial v}{\partial\bm n}})_e.
\end{equation}
Similarly, using integration by parts again, we deduce
 \begin{equation}\label{oneq:cst4}
-\lambda(\Delta w,v)_\varOmega=\lambda(\nabla w,\nabla v)_\varOmega
\end{equation}
and
 \begin{equation}\label{oneq:cst5}
-\lambda\sum_{K\in\mathcal K_h}(w,\Delta v)_K=\lambda(\nabla w,\nabla v)_\varOmega-
\lambda\sum_{e\in\mathcal E_h^\circ}(w,\jp{\frac{\partial v}{\partial\bm n}})_e.
\end{equation}
The conclusion can now be demonstrated by using \eqref{oneq:cst3}-\eqref{oneq:cst5} and some basic algebraic manipulation.
\end{proof}

With the help of the above lemma, we have the following estimate for $\varXi_2^\cc(u_I^\cc)$.
\begin{lemma}\label{lem:errC2}
If the true solution $u\in H_0^1(\varOmega)\cap H^s(\varOmega)$ with $s>\frac52$, then we have
\begin{equation}\label{ineq:errC2}
  \varXi_2^\cc(u_I^\cc)\leq Ch^{\sigma_{k,s}}\|u\|_{s,\varOmega}.
\end{equation}
\end{lemma}
\begin{proof}
  Applying Lemma \ref{lem:cst}, we get
  \begin{equation}\label{ineq:errC2-1}
    B_{\frac12}(u,w_h)=(L_\lambda u,L_\lambda w_h)_\varOmega.
  \end{equation}
  Since $F_\gamma[u]=0$ $\mathrm{a.e.}$ in $\Omega$, we have
  \begin{equation}\label{ineq:errC2-2}
  \begin{split}
  &\quad A_h^\cc(u_I^\cc;w_h)+B_{\frac12}(u,w_h) \\
  & = \sum_{K\in\mathcal K_h}\Big((\widehat F_\gamma^\cc[u_I^\cc]-\widehat L_\lambda^\cc u_I^\cc,\widehat L_\lambda^\cc w_h)_K+S^{K,\cc}(u_I^\cc-\Pi_{k,\Hs}^{K,\cc}u_I^\cc,w_h-\Pi_{k,\Hs}^{K,\cc}w_h)\Big)\\
  &\qquad+(L_\lambda u,L_\lambda w_h)_\varOmega\\
   &=\sum_{K\in\mathcal K_h}(\widehat F_\gamma^\cc[u_I^\cc]-\widehat F_\gamma^\cc[u]-\widehat L_\lambda^\cc (u_I^\cc-u),\widehat L_\lambda^\cc w_h)_K\\
   &\quad+\sum_{K\in\mathcal K_h}(\widehat F_\gamma^\cc[u]- F_\gamma[u]-(\widehat L_\lambda^\cc u-L_\lambda u),\widehat L_\lambda^\cc w_h)_K\\
   &\quad+\sum_{K\in\mathcal K_h}(L_\lambda u,L_\lambda w_h-\widehat L_\lambda^\cc w_h)_K\\
   &\quad+\sum_{K\in\mathcal K_h} S^{K,\cc}(u_I^\cc-\Pi_{k,\Hs}^{K,\cc}u_I^\cc,w_h-\Pi_{k,\Hs}^{K,\cc}w_h))\\
   &:=T_1+T_2+T_3+T_4.
  \end{split}
  \end{equation}
Now we estimate $T_i$, $i=1,\cdots,4$. For $T_1$, we use the Cauchy-Schwarz inequality, an argument similar as in the proof of Lemma \ref{lem:cordes}, Lemma \ref{lem:appc},  and the stability of $L^2$ projection, we  deduce
  \begin{equation}\label{ineq:errC2-T1}
  \begin{split}
  |T_1|&\leq C\left[\sum_{K\in\mathcal K_h}\Big(|u-u_I^\cc|_{2,K}^2+2\lambda|u-u_I^\cc|_{1,K}^2+\lambda^2\|u-u_I^\cc\|_{0,K}^2\Big)\right]^{\frac12}\\
  &\qquad\qquad\cdot\left(\sum_{K\in\mathcal K_h}\|\widehat L_\lambda^\cc w_h\|_{0,K}^2\right)^{\frac12}\\
  &\leq Ch^{\sigma_{k,s}}\|u\|_{s,\varOmega}\norm{w_h}_{B,h}.
 \end{split}
  \end{equation}
  For $T_2$, using, the Cauchy-Schwarz inequality and the standard approximation theory, we get
    \begin{equation}\label{ineq:errC2-T2}
  \begin{split}
  |T_2|&\leq C\inf_{p\in\mathbb P_k(\mathcal K_h)}\left(\sum_{K\in\mathcal K_h}\|u-p\|_{2,K}^2\right)^{\frac12}\cdot\left(\sum_{K\in\mathcal K_h}\|\widehat L_\lambda^\cc w_h\|_{0,K}^2\right)^{\frac12}\\
  &\leq Ch^{\sigma_{k,s}}\|u\|_{s,\varOmega}\norm{w_h}_{B,h}.
 \end{split}
  \end{equation}
For $T_3$, using orthogonality and standard approximation theory, we have:
    \begin{equation}\label{ineq:errC2-T3}
  \begin{split}
  |T_3|&=\sum_{K\in\mathcal K_h}(L_\lambda u,(I-\Pi_{k-2}^{K,\cc})\Delta w_h)_K-\sum_{K\in\mathcal K_h}(L_\lambda u, \lambda((I-\Pi_{k-2}^{K,\cc})w_h))_K\\
  &=\sum_{K\in\mathcal K_h}((I-\Pi_{k-2}^{K,\cc})L_\lambda u,\Delta w_h)_K-\sum_{K\in\mathcal K_h}((I-\Pi_{k-2}^{K,\cc})L_\lambda u, \lambda w_h)_K \\
  &\leq Ch^{\sigma_{k,s}}\|u\|_{s,\varOmega}\norm{w_h}_{B,h}.
 \end{split}
  \end{equation}
Finally, for $T_4$, using \eqref{ineq:Ss}, adding and subtracting $u$, $\Pi_{k,\Hs}^{K,\cc}u$, we obtain
\begin{equation}\label{ineq:errC2-T4}
  |T_4|\leq Ch^{\sigma_{k,s}}\|u\|_{s,\varOmega}\norm{w_h}_{B,h}.
  \end{equation}
  Substituting \eqref{ineq:errC2-T1}-\eqref{ineq:errC2-T4} into \eqref{ineq:errC2-2}, we arrive at
  \[|A_h(u_I^\cc;w_h)+B_{\frac12}(u,w_h)|\leq  Ch^{\sigma_{k,s}}\|u\|_{s,\varOmega}\norm{w_h}_{B,h},\]
  which completes the proof.
\end{proof}

\begin{lemma}\label{lem:errC3}
Let $u\in H_0^1(\varOmega)\cap H^s(\varOmega)$ be the solution of the problem \eqref{eq:HJB} with $s>2$.  Then, the following estimate holds
\begin{equation}\label{ineq:errC3}
  \inf_{p\in\mathbb P_k(\mathcal K_h)}\varXi_3^\cc(p)\leq Ch^{\sigma_{k,s}}\|u\|_{s,\varOmega}.
\end{equation}
\end{lemma}
\begin{proof}
For any $K\in\mathcal K_h$, we take $p\in\mathbb P_k(\mathcal K_h)$ and $p|_K=\Pi_k^{K,\cc}u$.  From the standard approximation theory, we have
  \begin{equation}\label{ineq:errC3-1}
  \norm{u-p}_{B,h}\leq Ch^{\sigma_{k,s}}\|u\|_{s,\varOmega}.
\end{equation}
We now  focus on estimating the second term  of $\varXi_3^\cc(p)$, and we first derive for each $K\in\mathcal K_h$ and $w_h\in V_{h,k}^\cc$ that
\begin{equation}\label{ineq:errC3-2}
\begin{split}
  &(D^2\Pi_k^{K,\cc}u,D^2w_h)_K-(\Pi_{k-2}^{K,\cc}D^2\Pi_k^{K,\cc}u,\Pi_{k-2}^{K,\cc}D^2w_h)_K\\
  =~&(D^2\Pi_k^{K,\cc}u,D^2w_h)_K-(D^2\Pi_k^{K,\cc}u,\Pi_{k-2}^{K,\cc}D^2w_h)_K\\
  =~&(D^2\Pi_k^{K,\cc}u,(I-\Pi_{k-2}^{K,\cc})D^2w_h)_K\\
  =~&((I-\Pi_{k-2}^{K,\cc})D^2\Pi_k^{K,\cc}u,D^2w_h)_K\\
  \leq~& Ch^{\sigma_{k,s}}\|u\|_{s,K}|w_h|_{2,K}.
  \end{split}
\end{equation}
Similarly, we  obtain
\begin{equation}\label{ineq:errC3-3}
(\nabla\Pi_k^{K,\cc}u,\nabla w_h)_K-(\Pi_{k-2}^{K,\cc}\nabla\Pi_k^{K,\cc}u,\Pi_{k-2}^{K,\cc}\nabla w_h)_K\leq Ch^{\sigma_{k,s}}\|u\|_{s,K}|w_h|_{1,K},
\end{equation}
and
\begin{equation}\label{ineq:errC3-4}
(\Pi_k^{K,\cc}u, w_h)_K-(\Pi_{k-2}^{K,\cc}\Pi_k^{K,\cc}u,\Pi_{k-2}^{K,\cc} w_h)_K\leq Ch^{\sigma_{k,s}}\|u\|_{s,K}\|w_h\|_{0,K}.
\end{equation}
Combing \eqref{ineq:errC3-2}, \eqref{ineq:errC3-3} and \eqref{ineq:errC3-4} leads to
\begin{equation}\label{ineq:errC3-5}
|B_*^K(p,w_h)-B_{h,*}^{K,\cc}(p,w_h)|\leq Ch^{\sigma_{k,s}}\|u\|_{s,K}\norm{w_h}_{B,h,K}.
\end{equation}
Using the same argument as above, one can show that
\begin{equation}\label{ineq:errC3-6}
|(L_\lambda p,L_\lambda w_h)_K-(\widehat L_\lambda^\cc p,\widehat L_\lambda^\cc w_h)_K|\leq Ch^{\sigma_{k,s}}\|u\|_{s,K}\norm{w_h}_{B,h,K}.
\end{equation}
Finally, we obtain the desired result from \eqref{ineq:errC3-5}, \eqref{ineq:errC3-6},  and the Cauchy-Schwarz inequality.
\end{proof}

By employing Lemmas \ref{lem:errC1}, \ref{lem:errC2}, \ref{lem:errC3}, and Lemma \ref{lem:errorbound}, we derive the following \textit{a priori} error estimate.
\begin{theorem}\label{thm:errC}
  Let $u\in H^s(\varOmega)\cap H_0^1(\varOmega)$ with $s>\frac52$ be the solution of the problem \eqref{eq:HJB}, and let $u_h^\cc$ be the solution of the discrete problem \eqref{eq:c1vem}. Then, the following estimate holds:
  \begin{equation}\label{ineq:errconforming}
    \norm{u-u_h^\cc}_{B,h}\leq Ch^{\sigma_{k,s}}\|u\|_{s,\varOmega}.
  \end{equation}
\end{theorem}

\section{The nonconforming virtual element method}\label{sec:nonconforming}
In this section, we further investigate the $C^0$-nonconforming virtual element method for the HJB equation \eqref{eq:HJB}. We first recall the local $C^0$-nonconforming virtual element space as introduced in \cite{AABMR2013,ZCZ2016,LZWC2024}, defined as follows:
\begin{align*}
V_{h,k}^{K,\nc}:=\Bigg\{v\in H^2(K):\,&\Delta^2v\in\mathbb P_k(K), \Delta v|_e\in\mathbb P_{k-2}(e), e\subset\mathcal E_h\cap\partial K, \\
&(\Pi_{k,\Hs}^{K,\nc} v-v,q^\circ)_K=0\quad\forall q^\circ\in\bigcup_{j=k-3}^k\mathscr M_j^\circ(K)\Bigg\}.
\end{align*}
Here, the $H^2$ projection operator
$\Pi_{k,\Hs}^{K,\nc}:\widetilde V_{h,k}^{K,\nc}\to\mathbb P_k(K)\subset \widetilde V_{h,k}^{K,\nc}$ is defined by:\[
\left\{
    \begin{array}{l}
       (D^2\Pi_{k,\Hs}^{K,\nc}\psi,D^2q)_K=(D^2\psi,D^2q)_K\qquad \forall q\in\mathbb P_k(K)\\
       \\
       \widehat{\Pi_{k,\Hs}^{K,\nc}\psi} = \widehat \psi,\quad\widehat{\nabla\Pi_{k,\Hs}^{K,\nc}\psi} = \widehat{\nabla \psi},
    \end{array}
\right.
\]
for any $\psi\in\widetilde V_{h,k}^{K,\nc}$, and
\[\widetilde V_{h,k}^{K,\nc}:=\bigg\{v\in H^2(K): \Delta^2v\in\mathbb P_k(K), \Delta v|_e\in\mathbb P_{k-2}(e), e\subset\mathcal E_h\cap\partial K
\bigg\}. \] The projection $\Pi_{k,\Hs}^{K,\nc}\psi$  is computable,  and the space $V_{h,k}^{K,\nc}$ can be uniquely characterized using the following degrees of freedom, as detailed in \cite{ZCZ2016,LZWC2024}:
\begin{itemize}
  \item The value of $v(\xi)$, for any vertex $\xi$ in $K$;
  \item The moments $\frac1{h_e}\int_eqv$, for any $q\in\mathscr M_{k-2}(e)$ and $e\subset\partial K$;
  \item The moments $\int_eq\frac{\partial v}{\partial\bm n}$, for any $q\in\mathscr M_{k-2}(e)$ and $e\subset\partial K$;
  \item For $k\geq 4$, the moments $\frac1{h_K^2}\int_Kqv$, for any $q\in \mathscr M_{k-4}(K)$.
\end{itemize}
Figure \ref{fig:dofnc} illustrates the degrees of freedom for the first two low-order elements with $k = 2$ and $k = 3$.
\begin{figure}
   \centering
  \includegraphics[width=0.8\textwidth]{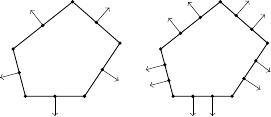}
  \caption{The degrees of freedom with $k=2$ (left) and $k=3$ (right).}
   \label{fig:dofnc}
\end{figure}

The global $C^0$-nonconforming virtual element space $V_{h,k}^\nc$ is defined as follows:
\[V_{h,k}^\nc:=\{v\in H_k^{2,\nc}(\mathcal K_h):v|_K\in V_{h,k}^{K,\nc}~~\forall K\in\mathcal K_h\}.\]

To define the nonlinear and bilinear forms for the nonconforming virtual finite element method, we first consider the $L^2$ projection operator.
 For any function $v\in V_{h,k}^{K,\nc}$, the $L^2$ projection $\Pi_k^{K,\nc}v$ of $v$ onto $\mathbb P_k(K)$ is computable by the definition of the space $V_{h,k}^{K,\nc}$. Furthermore, by utilizing the identities \eqref{eq:cp1} and \eqref{eq:cp2}, it follows that $\Pi_{k-2}^{K,\nc}D^2v$ and $\Pi_{k-1}^{K,\nc}\nabla v$  are also computable.

Analogous to the conforming case, we can define $\widehat F_\gamma^\nc[\cdot]$, $\widehat L_\lambda^\nc$, $\dof_j^\nc$ and $B_{h,\frac12}^{K,\nc}(\cdot,\cdot)$. We now define
\[A_h^{K,\nc}(w_h;v_h):=(\widehat F_\gamma^\nc[w_h]-\widehat L_\lambda^\nc w_h,\widehat L_\lambda^\nc v_h)_K+S^{K,\nc}(u_h-\Pi_{k,\Hs}^{K,\nc}u_h,v_h-\Pi_{k,\Hs}^{K,\nc}v_h),\]
where $S^{K,\nc}(\cdot, \cdot)$ is a symmetric and positive definite bilinear form satisfying
\begin{equation}\label{ineq:Ssnc}
  c_\star\norm{v_h-\Pi_{k,\Hs}^{K,\nc}v_h}_{B,h,K}^2\leq S^{K,\nc}(v_h-\Pi_{k,\Hs}^{K,\nc}v_h,v_h-\Pi_{k,\Hs}^{K,\nc}v_h) \leq c^\star\norm{v_h-\Pi_{k,\Hs}^{K,\nc}v_h}_{B,h,K}^2,
\end{equation}
for any $v_h\in V_h^{K,\nc}$.
 Specifically, we choose
\[S^{K,\nc}(u_h,v_h):=(h_K^{-2}+2\lambda+\lambda^2h_K^2)\sum_{j=1}^{n_K^\nc}\dof_j^\nc(u_h)\cdot\dof_j^\nc(v_h),\]
where $n_K^\nc$ denotes the number of degrees of freedom in element $K$.

Finally, we define the nonlinear form
\[a_h^{K,\nc}(w_h;v_h):=A_h^{K,\nc}(w_h;v_h)+B_{h,\frac12}^{K,\nc}(w_h,v_h).\]
The $C^0$-nonconforming virtual element scheme for solving the HJB equation is then formulated as: Find $u_h^\nc\in V_{h,k}^\nc$ such that:
\begin{equation}\label{eq:c0vem}
  a_h^\nc(u_h^\nc;v_h)=0\qquad\forall v_h\in V_{h,k}^\nc.
\end{equation}

By employing an argument analogous to Lemma \ref{lem:e&uc}, we can prove the following lemma
\begin{lemma}
The problem \eqref{eq:c0vem} admits a unique solution.
\end{lemma}

To proceed with the error estimation, we first recall the canonical interpolation  $v_I^\nc|_K\in V_{h,k}^{K,\nc}$ for a given smooth function  $v$, which satisfies
\[\dof_j^\nc(v-v_I^\nc)=0,\quad j=1,\cdots, n_K^\nc.\]
Under the mesh assumption \ref{ams:mesh}, we then recall the following approximation property from \cite{ZCZ2016,LZWC2024}.
 \begin{lemma}\label{lem:appnc}
For $m=0,1,2$, $2< s\leq k+1$, and $v\in H^s(K)$, we have
\begin{equation}\label{ineq:appnc}
  \|v-v_I^\nc\|_{m,K}\leq Ch^{s-m}\|v\|_{s,K}.
\end{equation}
\end{lemma}

Our main result for the nonconforming virtual element methods can be stated as follows:
\begin{theorem}\label{thm:errnC}
  Let $u$ and $u_h^\nc$ be the solution of the problem \eqref{eq:HJB} and  \eqref{eq:c0vem}, respectively. Assume that $u\in H_0^1(\varOmega)
  \cap H^s(\varOmega)$ with $s>\frac52$. Then, the following error estimate holds:
  \begin{equation}\label{ineq:errnonconforming}
    \norm{u-u_h^\nc}_{B,h}\leq Ch^{\sigma_{k,s}}\|u\|_{s,\varOmega}.
  \end{equation}
\end{theorem}
\begin{proof}
  According to Lemma \ref{lem:errorbound}, we have
  \begin{equation}\label{ineq:errnonconforming-1}
  \norm{u-u_h^\nc}_{B,h}\leq C\left[\inf_{v_h\in V_{k,h}^\nc}\Big(\varXi_1^\nc(v_h)+\varXi_2^\nc(v_h)\Big)+\inf_{p\in\mathbb P_k(\mathcal K_h)}\varXi_3^\nc(p)\right],
\end{equation}
with
\[\varXi_1^\nc(v_h):=\norm{u-v_h}_{B,h},\]
\[\varXi_2^\nc(v_h):=\sup_{w_h\in V_{h,k}^\nc\setminus\{0\}}\frac{|A_h^\nc(v_h;w_h)+B_{\frac12}(u,w_h)|}{\norm{w_h}_{B,h}},\]
and
\[
  \varXi_3^\nc(p):=\norm{u-p}_{B,h}+\sup_{w_h\in V_{h,k}^\nc\setminus\{0\}}\frac1{\norm{w_h}_{B,h}}\sum_{K\in\mathcal K_h}|B_{\frac12}^K(p,w_h)-B_{h,\frac12}^{K,\nc}(p,w_h)|.
  \]

By utilizing Lemma \ref{lem:appnc} and employing similar techniques as in the conforming case described in Lemma \ref{lem:errC3}, we obtain:
  \begin{equation}\label{ineq:errnonconforming-2}
   \varXi_1^\nc(u_I^\nc)+\inf_{p\in\mathbb P_k(\mathcal K_h)}\varXi_3^\nc(p)\leq Ch^{\sigma_{k,s}}\|u\|_{s,\varOmega}.
  \end{equation}

  Next, we deal with the term $\varXi_2^\nc(u_I^\nc)$.  Using the consistency result from Lemma \ref{lem:cst}, we derive:
    \begin{equation}\label{ineq:errnonconforming-3}
  \begin{split}
  &\quad A_h(u_I^\nc;w_h)+B_{\frac12}(u,w_h) \\
  & = \sum_{K\in\mathcal K_h}\Big((\widehat F_\gamma^\nc[u_I^\nc]-\widehat L_\lambda^\nc u_I^\nc,\widehat L_\lambda^\nc w_h)_K+S^{K,\nc}(u_I^\nc-\Pi_{k,\Hs}^{K,\nc}u_I^\nc,w_h-\Pi_{k,\Hs}^{K,\nc}w_h)\Big)\\
  &\qquad+(L_\lambda u,L_\lambda w_h)_\varOmega+\frac12\mathscr N(u,w_h)\\
  &:=\mathbb T_1+\frac12\mathscr N(u,w_h).
  \end{split}
  \end{equation}
As in the proof of Lemma \ref{lem:errC2}, we can show that
\begin{equation}\label{ineq:errnonconforming-4}
   |\mathbb T_1|\leq Ch^{\sigma_{k,s}}\|u\|_{s,\varOmega}\norm{w_h}_{B,h}.
\end{equation}
  For the term $\mathscr N(u,w_h)$, we use the fact:
  \[\int_e\jp{\frac{\partial w_h}{\partial\bm n}}q=0\quad\forall q\in \mathbb P_{k-2}(e),\,\,\forall e\in\mathcal E_h^\circ,\]
to get:
  \begin{equation}\label{ineq:errnonconforming-5}
  \begin{split}
   &\int_e(\frac{\partial^2u}{\partial\bm t^2}+\lambda u)\jp{\frac{\partial w_h}{\partial\bm n}}\\
   =\,&\int_e\Big((\frac{\partial^2u}{\partial\bm t^2}+\lambda u)-\mathscr P_{k-2}^e(\frac{\partial^2u}{\partial\bm t^2}+\lambda u)\Big)
   \jp{\frac{\partial w_h}{\partial\bm n}
    -\mathscr P_0^e\left(\frac{\partial w_h}{\partial\bm n}\right)}\\
    \leq\,&\left\|(\frac{\partial^2u}{\partial\bm t^2}+\lambda u)-
    \mathscr P_{k-2}^e(\frac{\partial^2u}{\partial\bm t^2}+\lambda u)\right\|_{0,e}
    \left\|\jp{\frac{\partial w_h}{\partial\bm n}
    -\mathscr P_0^e\left(\frac{\partial w_h}{\partial\bm n}\right)}\right\|_{0,e},
  \end{split}
  \end{equation}
for any $e\in\mathcal E_h^\circ$, where $\mathscr P_{k-2}^e$ denotes the $L^2$ projection  onto the polynomial space $\mathbb P_{k-2}(e)$.

 Applying standard approximation techniques for nonconforming errors (see, e.g., \cite{C1978, BS2008}), we obtain:
  \begin{equation}\label{ineq:errnonconforming-6}
 \left\|(\frac{\partial^2u}{\partial\bm t^2}+\lambda u)-
    \mathscr P_{k-2}^e(\frac{\partial^2u}{\partial\bm t^2}+\lambda u)\right\|_{0,e}\leq Ch^{\sigma_{k,s}-\frac12}\|u\|_{s,K^l\cup K^r},
  \end{equation}
    \begin{equation}\label{ineq:errnonconforming-7}
\left\| \jp{\frac{\partial w_h}{\partial\bm n}
    -\mathscr P_0^e\left(\frac{\partial w_h}{\partial\bm n}\right)}\right\|_{0,e}\leq Ch^{\frac12}\Big(\|w_h\|_{2,K^l}^2+\|w_h\|_{2,K^r}^2\Big)^{\frac12},
  \end{equation}
  for the edge $e$ shared by elements $K^l$ and $K^r$. Consequently, from \eqref{ineq:errnonconforming-5}, \eqref{ineq:errnonconforming-6}, and \eqref{ineq:errnonconforming-7}, we find:
  \begin{equation}\label{ineq:errnonconforming-8}
   |\mathscr N(u,w_h)|\leq Ch^{\sigma_{k,s}}\|u\|_{s,\varOmega}\norm{w_h}_{B,h}.
  \end{equation}
The proof is completed by combining \eqref{ineq:errnonconforming-3} with \eqref{ineq:errnonconforming-4} and \eqref{ineq:errnonconforming-8}.
\end{proof}

\section{Semismooth Newton's method for solving the HJB equation}\label{sec:SemismoothNewton}
Since the form $a_h^\cc$ (or $a_h^\nc$) is nonlinear,
in this section we follow a similar argument as \cite{SS2014} and use the semismooth Newton’s method introduced in \cite{UM2002} to solve the discrete equations.  We take $C^0$-nonconforming virtual element as an example to
illustrate the algorithm.

To apply the semismooth Newton’s method, for any $v\in H+V_{h,k}^\nc$, we define a set of admissible maximizers as
\begin{align*}
\varLambda[v]:=\Big\{g:\varOmega\to&\varLambda: g \text{ is measurable}; \\
&g(x)\in\mathop{\text{arg max}}\limits_{\alpha\in\varLambda} \Big(\A^\alpha:\Pi_{k-2}^{K,\nc}D^2v+\bm b^\alpha\cdot\Pi_{k-2}^{K,\nc}\nabla v-c^\alpha\Pi_{k-2}^{K,\nc}v-f^\alpha\Big)\\
&\forall x\in K\quad \text{a.e.} \text{ in } K, \quad \forall K\in\mathcal K_h\Big\}.
\end{align*}

For a given measurable function $\alpha:\varOmega\to \varLambda$, we define the bilinear form $a_h^{\nc,\alpha}:V_{h,k}^\nc\times V_{h,k}^\nc\to\mathbb R$ by
\begin{align*}a_h^{\nc,\alpha}(u,v):=(\gamma^\alpha\widehat L^{\alpha,\nc} u&-\widehat L_\lambda^\nc u,\widehat L_\lambda^\nc v)_\varOmega+B_{h,\frac12}^\nc(u,v)\\
&+\sum_{K\in\mathcal K_h}S^{K,\nc}(u-\Pi_{k,\Hs}^{K,\nc}u,v-\Pi_{k,\Hs}^{K,\nc}v),
\end{align*}
and define the linear functional $l_\alpha:V_{h,k}^\nc\to\mathbb R$ by
\[l_\alpha(v):=(\gamma^\alpha f^\alpha,\widehat L_\lambda^\nc v)_\varOmega.\]

The semismooth Newton's algorithm for solving equation \eqref{eq:c0vem} is outlined in the following Algorithm \ref{algo:Semismooth}.
\begin{algorithm}
\caption{Semismooth Newton’s method}
\label{algo:Semismooth}
\begin{algorithmic}[1]
\STATE{Given an initial guess $u_h^0\in V_{h,k}^\nc$, a tolerance $tol<1$ and an integer $itermax$}
\STATE{$j\leftarrow0$; $err\leftarrow1$}
\WHILE{$j<itermax$ and $err>tol$}
\STATE{Select an arbitrary $\alpha_j\in\varLambda[u_h^j]$}
\STATE{$u_h^{j+1}\leftarrow$ the solution of
\begin{equation}a_h^{\nc,\alpha_j}(u,v)=l_{\alpha_j}(v)\quad\forall v\in V_{h,k}^\nc \end{equation}}
\STATE{$err\leftarrow\norm{u_h^{j+1}-u_h^j}_{B,h}$}
\STATE{$u_h^j\leftarrow u_h^{j+1}$}
\STATE{$j\leftarrow j+1$}
\ENDWHILE
\RETURN $u_h^j$
\end{algorithmic}
\end{algorithm}

\section{Numerical results}\label{sec:numerical}
In this section, several numerical tests will be implemented to verify the convergence theory established in previous sections. 
For simplicity, we use the lowest order element ($k = 2$) to solve problems \eqref{eq:Ndiv} and \eqref{eq:HJB} in all tests. 
To alleviate notation, we denote
\[E_2:=\left(\sum_{K\in\mathcal K_h}\|D^2u-\Pi_0^KD^2u_h\|_{0,K}^2\right)^{\frac12},~~E_1:=\left(\sum_{K\in\mathcal K_h}\|\nabla u-\Pi_1^K\nabla u_h\|_{1,K}^2\right)^{\frac12},\]
\[E_0:=\left(\sum_{K\in\mathcal K_h}\|u-\Pi_2^Ku_h\|_{0,K}^2\right)^{\frac12},\]
where $u_h$ will be $u_h^\cc$ or $u_h^\nc$, and $\Pi_j^K$, $j=0,1,2$ will be $\Pi_j^{K,\cc}$ or $\Pi_j^{K,\nc}$  accordingly.

\subsection{Example 1} In this test, we consider the linear elliptic equations in nondivergence form on the computational domain $\varOmega=(0,1)^2$:
\[\A:D^2u+\bm b\cdot\nabla u-cu=f\quad\text{  in } \varOmega,\qquad u=0\quad\text{on }\partial\varOmega.\]
Here,  the matrix $\A$ is given by
\[\A=\begin{pmatrix}
    2 & 1 \\
    1 & 2
  \end{pmatrix},\]
and $\bm b=(x_1~~ x_2)^T$, $c=3$. We take the true solution $u=\sin(\pi x_1)\sin(\pi x_2)$ and the load term $f$ is determined by the exact solution.
 The Cordes condition holds for $\varepsilon=9/20$, since we can choose $\lambda=1$ and therefore
 \[\frac{|\A|^2+|\bm b|^2/(2\lambda)+(c/\lambda)^2}{(\tr\A+c/\lambda)^2}=\frac{38+x_1^2+x_2^2}{96}\leq\frac{20}{49}.\]
We first employ the triangle mesh ($\mathcal K^1$) and the square mesh ($\mathcal K^2$), as depicted in Figure \ref{fig:mesh1}, to test the convergence property of our formulations. The numerical errors are presented in Tables \ref{table:conformingtri}, \ref{table:nonconformingtri}, and Figure \ref{fig:ratesquare}. The numerical results illustrate the validity of the theoretical analysis, i.e. the convergence rate
is $\mathcal O(h)$ in the $H^2$ seminorm.  Moreover, we also find that both $u_h$ and $\nabla u_h$ converge at the rate of $\mathcal O(h^2)$.

Next, we conduct our numerical experiments on more general polygonal meshes: a structured hexagonal mesh ($\mathcal K^3$), a Voronoi mesh ($\mathcal K^4$),  and a randomly distorted quadrilateral mesh ($\mathcal K^5$),
as displayed in Figure \ref{fig:mesh2}. We present the numerical results in Tables \ref{fig:ratedualTri} through \ref{fig:rateperturbRect}. Similar convergence orders to those observed on the triangular mesh are noted.

\begin{figure}
    \centering
    \subfigure{
        \includegraphics[width=0.4\textwidth]{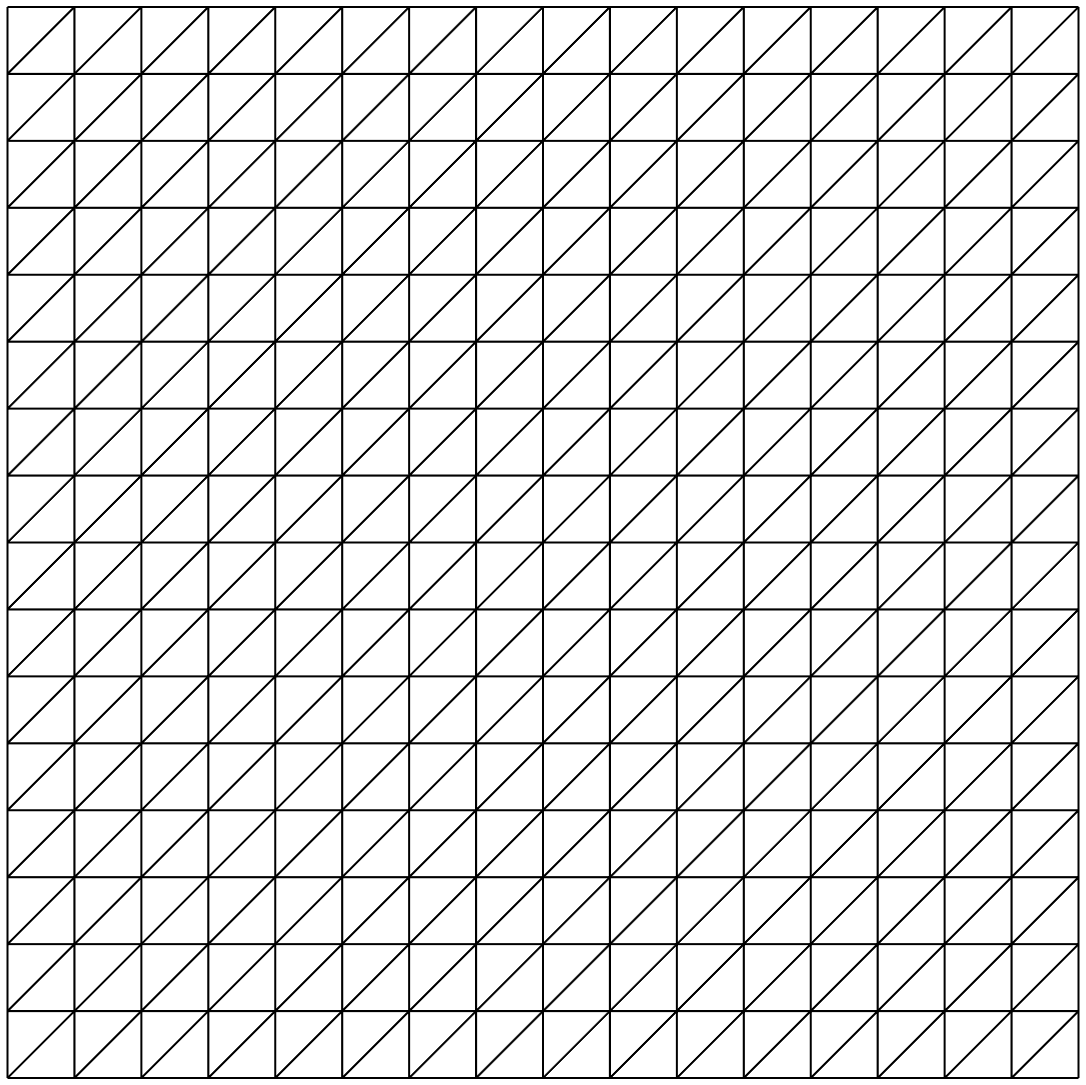}
    }
    \subfigure{
	\includegraphics[width=0.4\textwidth]{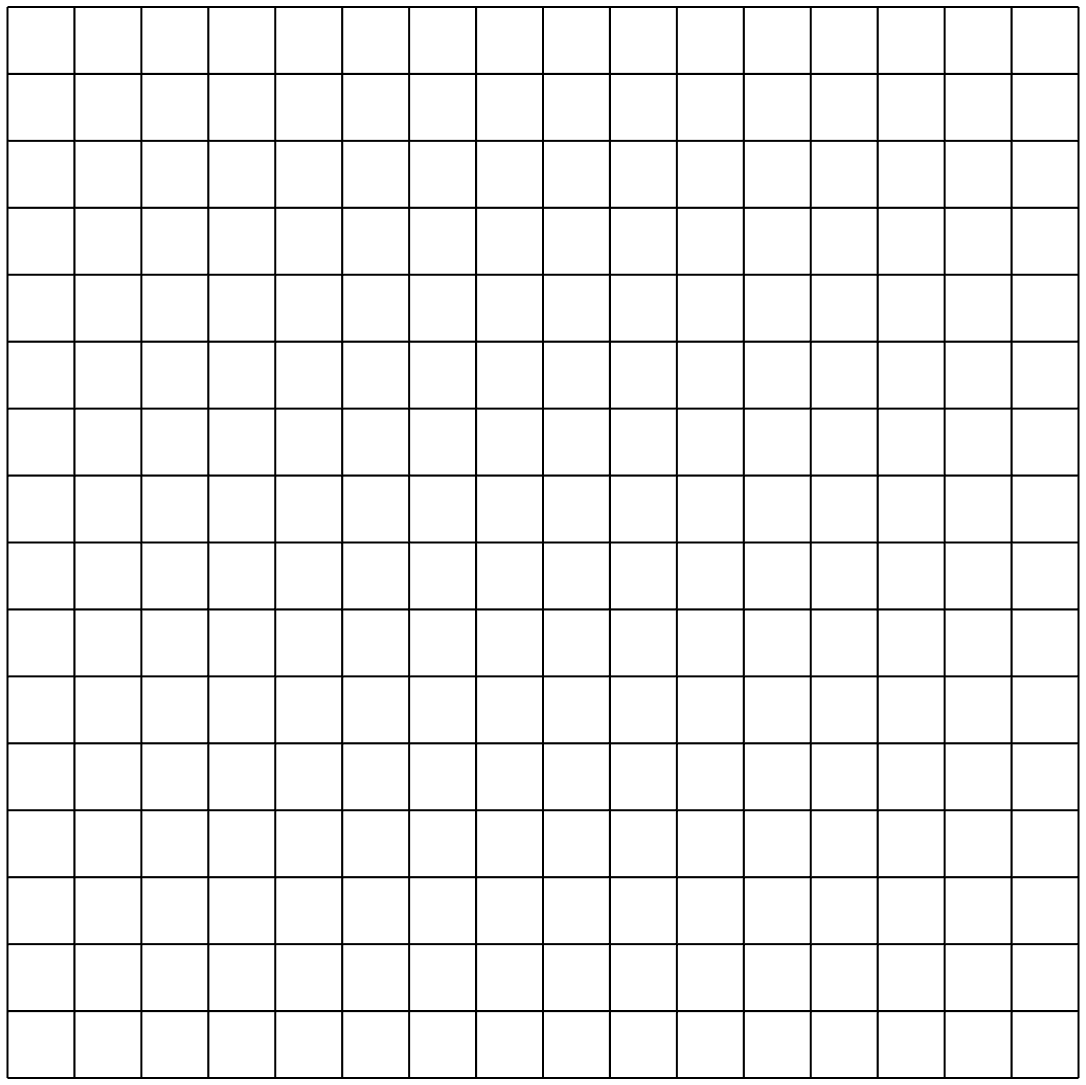}
    }
\caption{Meshes of $\mathcal K^1$ (left) and $\mathcal K^2$ (right)}
    \label{fig:mesh1}
\end{figure}

\begin{table}
\centering
\footnotesize
\caption{Errors for $C^1$-conforming VEM on the mesh $\mathcal K^1$}
\begin{tabular}{rcccccc}
\hline
 $1/h$ & $E_2$ & order& $E_1$ & order& $E_0$ & order\\ \toprule[1pt]
 8 &1.53e-00&$-$&3.50e-02&$-$&4.90e-03&$-$\\
 16 &7.68e-01&0.99&8.64e-03&2.02&1.19e-03&2.04\\
 32 &3.84e-01&1.00&2.15e-03&2.00&2.95e-04&2.01\\
 64 &1.92e-01&1.00&5.38e-04&2.00&7.37e-05&2.00\\
 128 &9.61e-02&1.00&1.34e-04&2.00&1.84e-05&2.00\\
 256 &4.81e-02&1.00&3.36e-05&2.00&4.60e-06&2.00\\ \hline
\end{tabular}
\label{table:conformingtri}
\end{table}

\begin{table}
\centering
\footnotesize
\caption{Errors for $C^0$-nonconforming VEM on the mesh $\mathcal K^1$ }
\begin{tabular}{rcccccc}
\hline
 $1/h$ & $E_2$ & order& $E_1$ & order& $E_0$ & order\\ \toprule[1pt]
 8 &1.72e-00&$-$&3.64e-02&$-$&4.93e-03&$-$\\
 16 &8.65e-01&1.00&9.10e-03&2.00&1.27e-03&1.95\\
 32 &4.33e-01&1.00&2.28e-03&2.00&3.21e-04&1.99\\
 64 &2.16e-01&1.00&5.69e-04&2.00&8.05e-05&2.00\\
 128 &1.08e-01&1.00&1.42e-04&2.00&2.01e-05&2.00\\
 256 &5.41e-02&1.00&3.56e-05&2.00&5.05e-06&2.00\\ \hline
\end{tabular}
\label{table:nonconformingtri}
\end{table}

\begin{figure}
    \centering
    \subfigure{
        \includegraphics[width=0.47\textwidth]{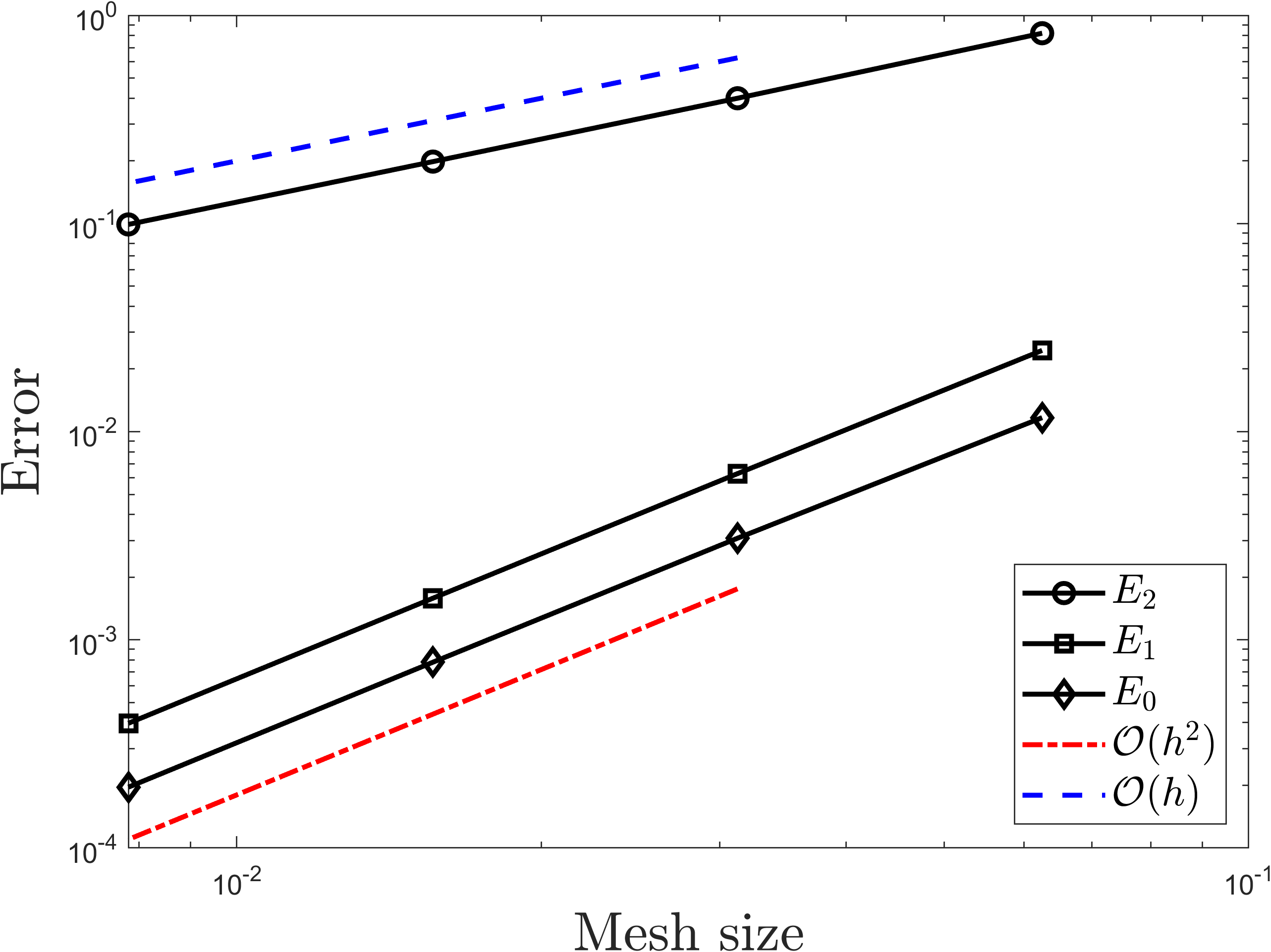}
    }
    \subfigure{
	\includegraphics[width=0.47\textwidth]{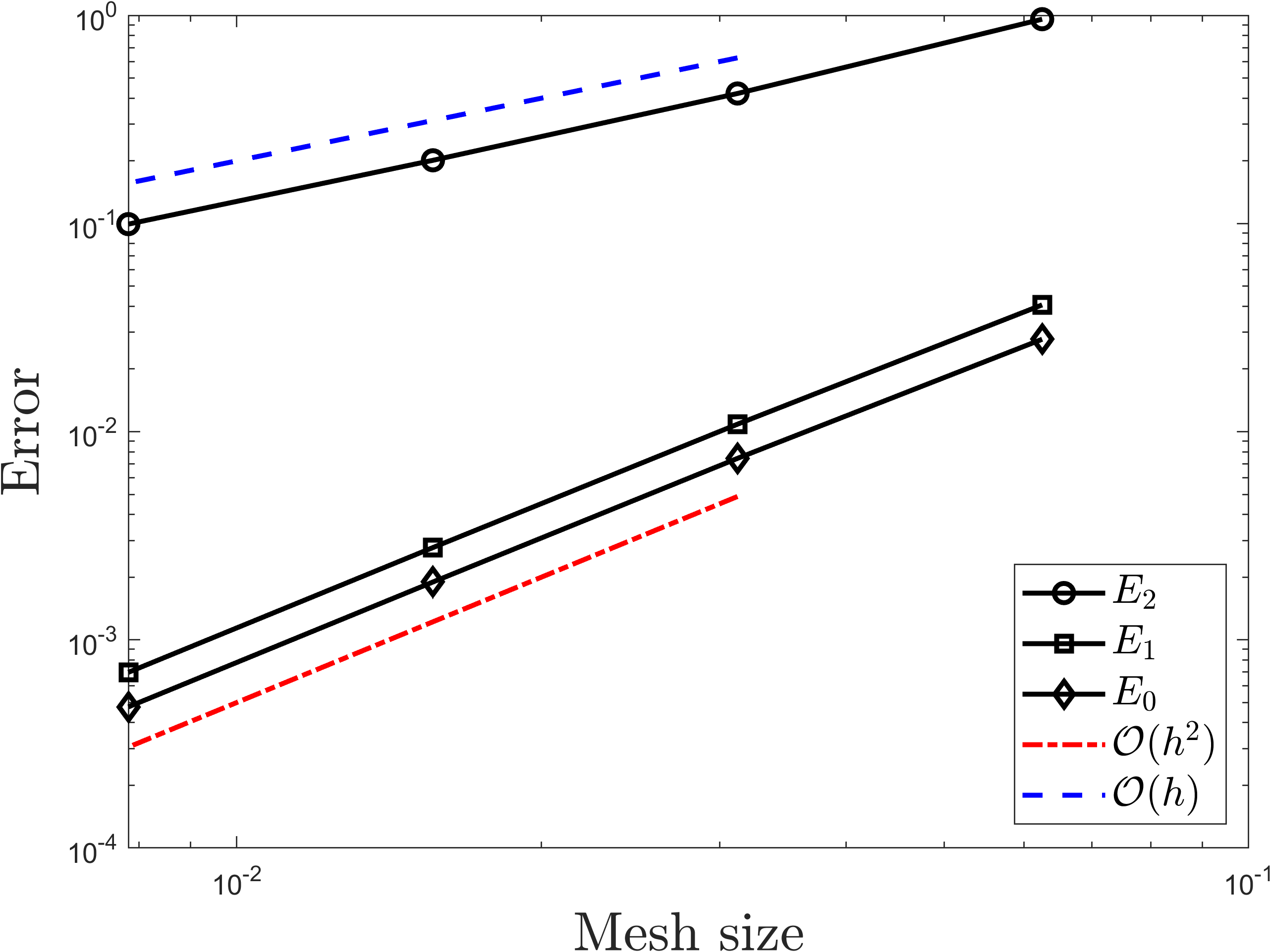}
    }
\caption{Errors for $C^1$-conforming (left) and $C^0$-nonconforming (right) VEM on the mesh $\mathcal K^2$}
    \label{fig:ratesquare}
\end{figure}

\begin{figure}
    \centering
    \subfigure{
        \includegraphics[width=0.3\textwidth]{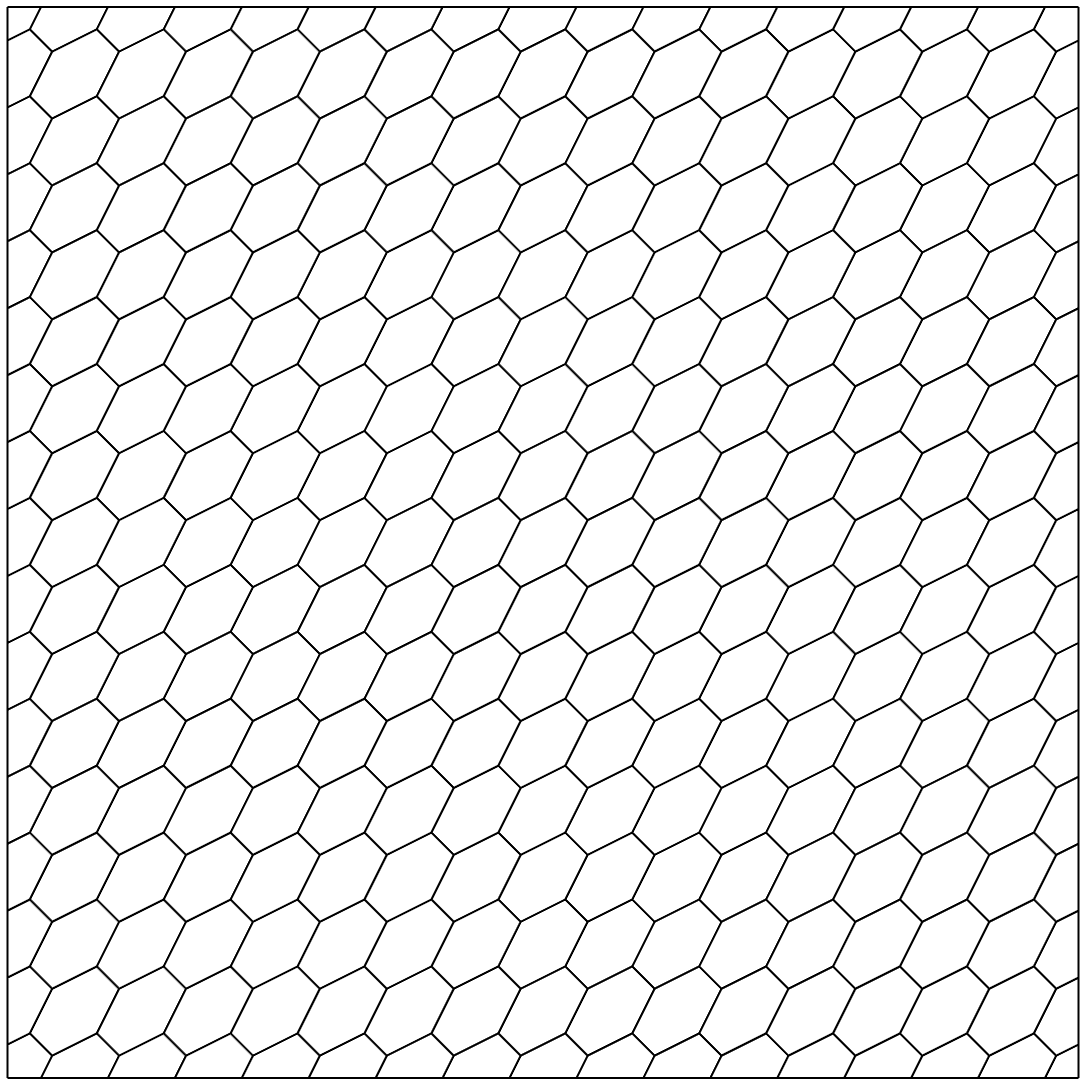}
    }
    \subfigure{
	\includegraphics[width=0.3\textwidth]{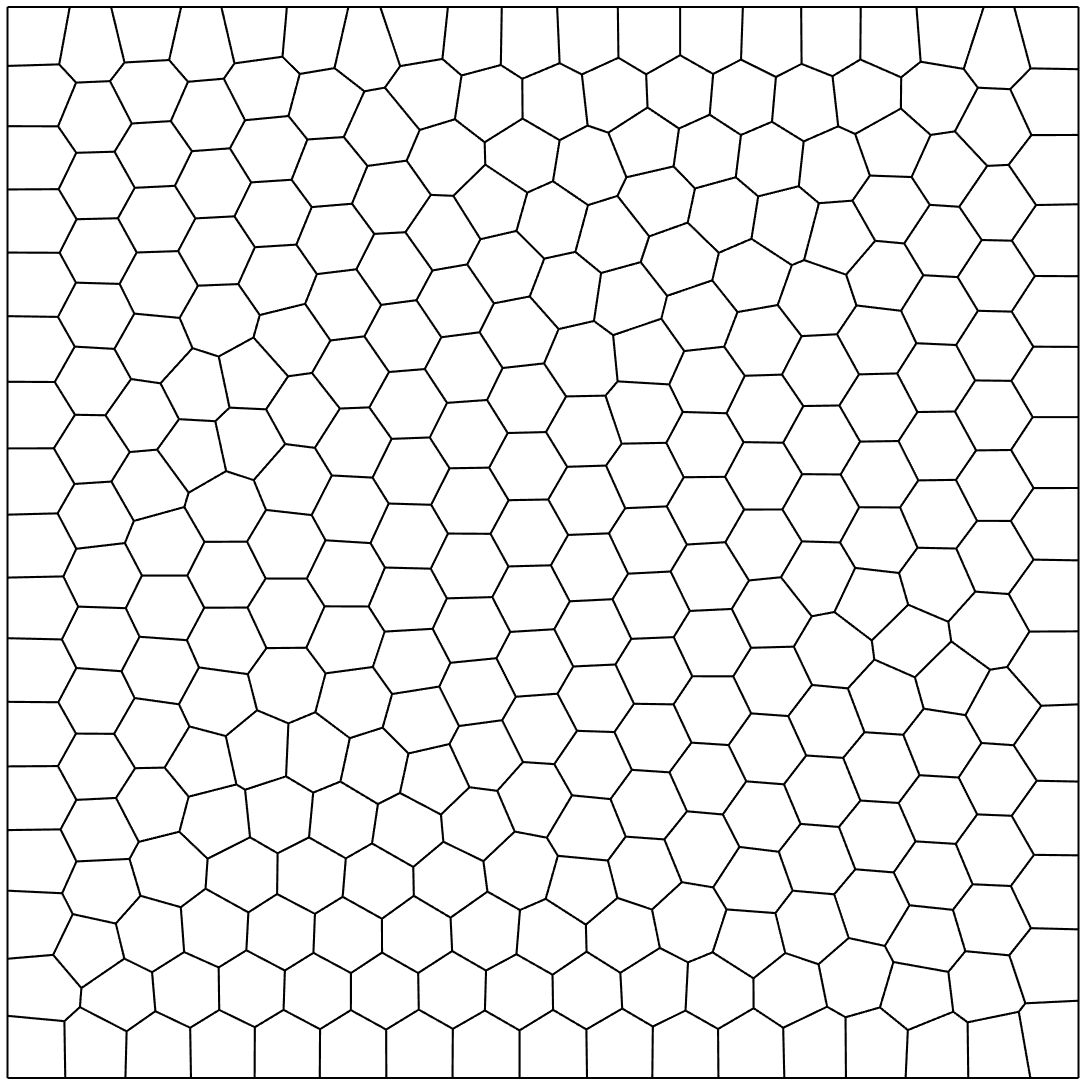}
    }
        \subfigure{
	\includegraphics[width=0.3\textwidth]{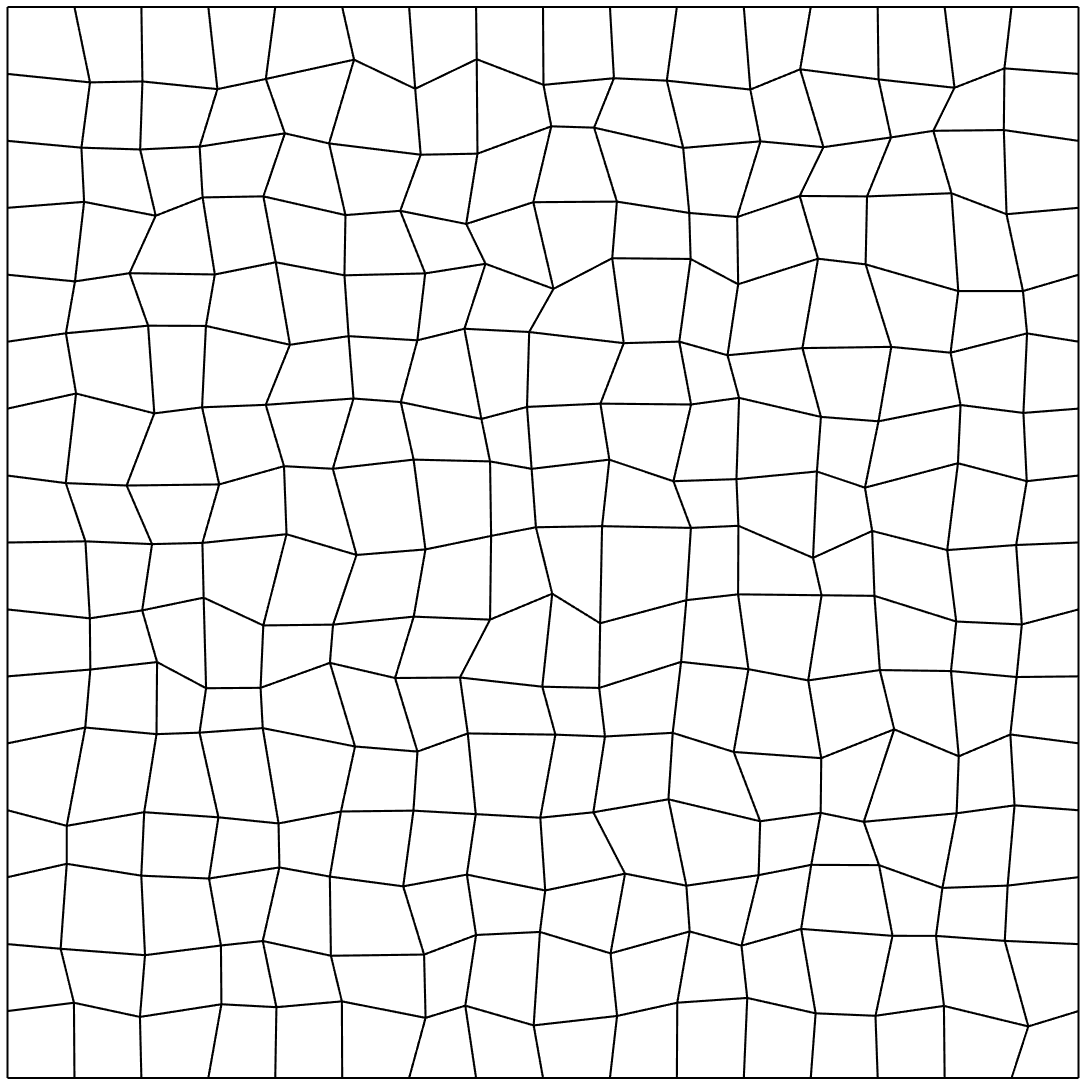}
    }
\caption{Meshes of $\mathcal K^3$ (left), $\mathcal K^4$ (middle) and $\mathcal K^5$ (right)}
    \label{fig:mesh2}
\end{figure}

\begin{figure}
    \centering
    \subfigure{
        \includegraphics[width=0.47\textwidth]{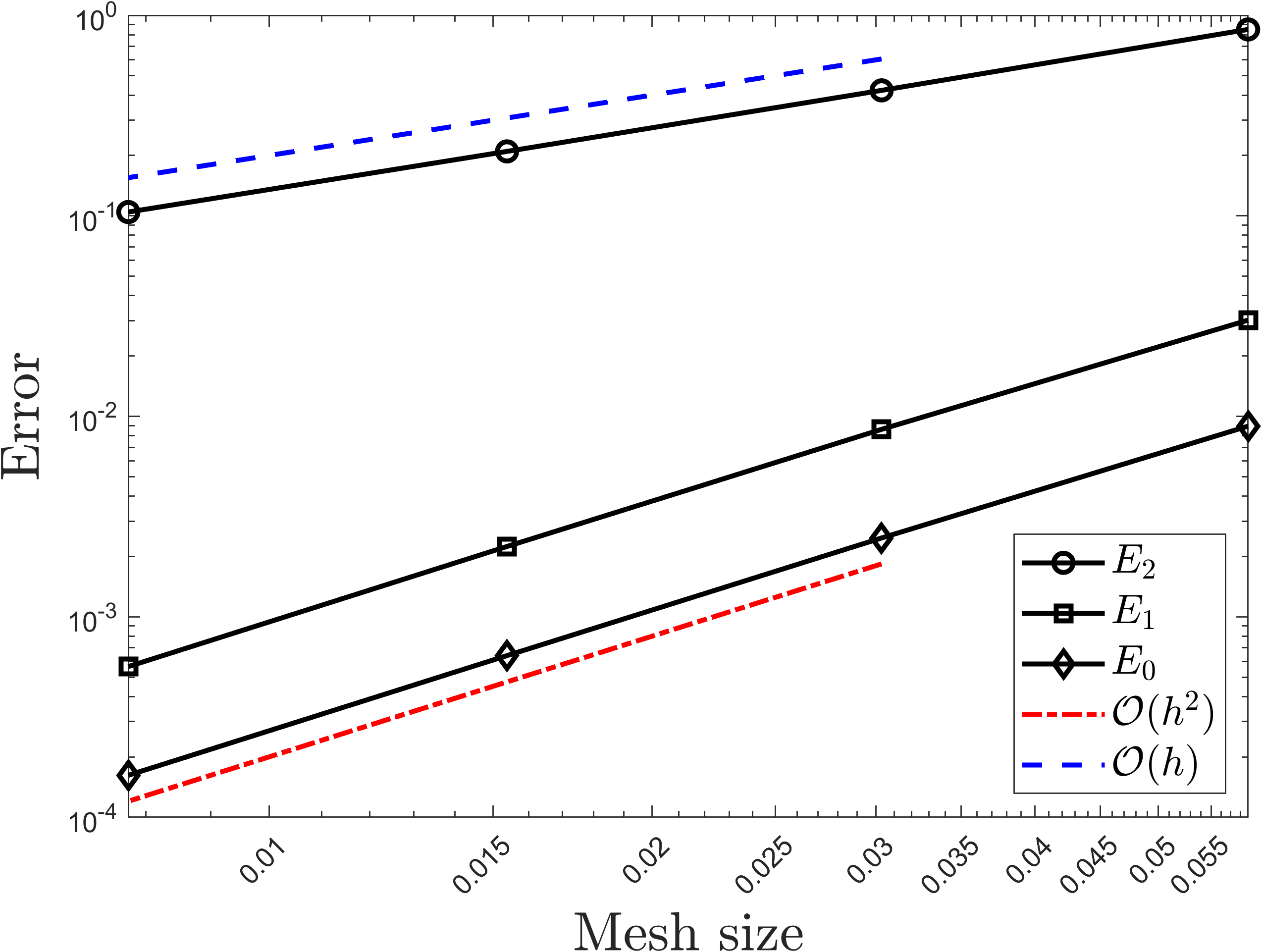}
    }
    \subfigure{
	\includegraphics[width=0.47\textwidth]{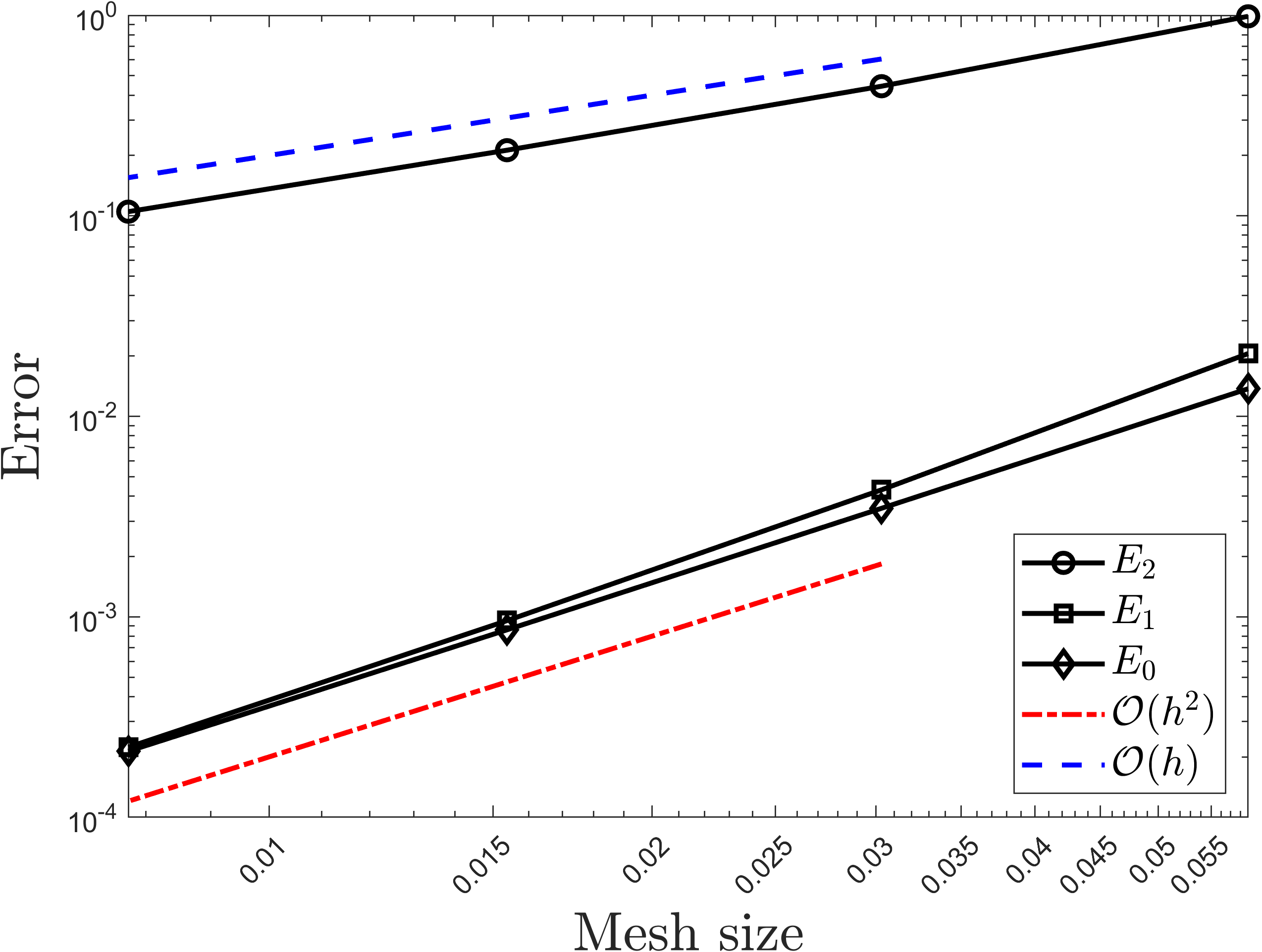}
    }
\caption{Errors for $C^1$-conforming (left) and $C^0$-nonconforming (right) VEM on the mesh $\mathcal K^3$}
    \label{fig:ratedualTri}
\end{figure}

\begin{figure}
    \centering
    \subfigure{
        \includegraphics[width=0.47\textwidth]{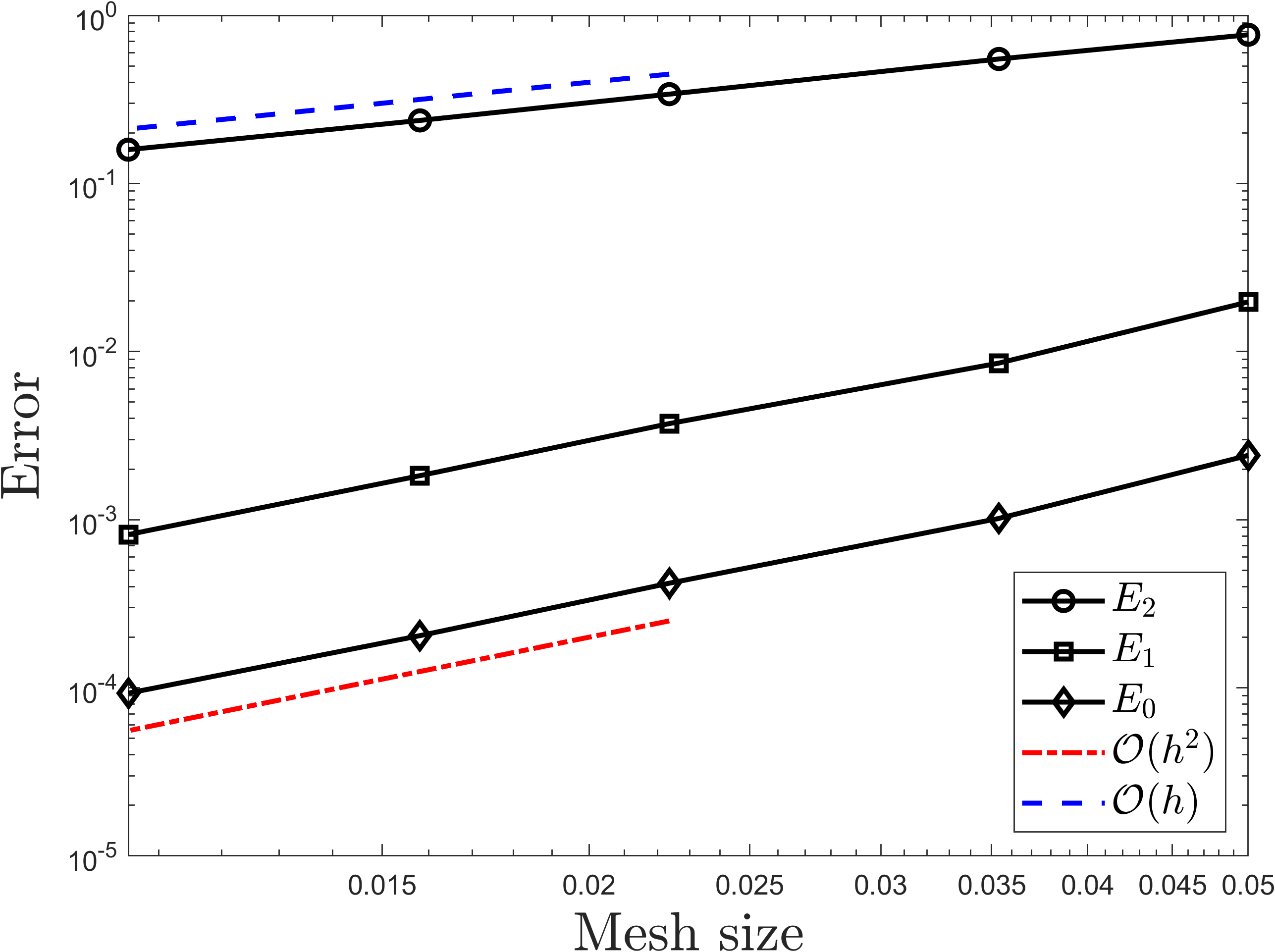}
    }
    \subfigure{
	\includegraphics[width=0.47\textwidth]{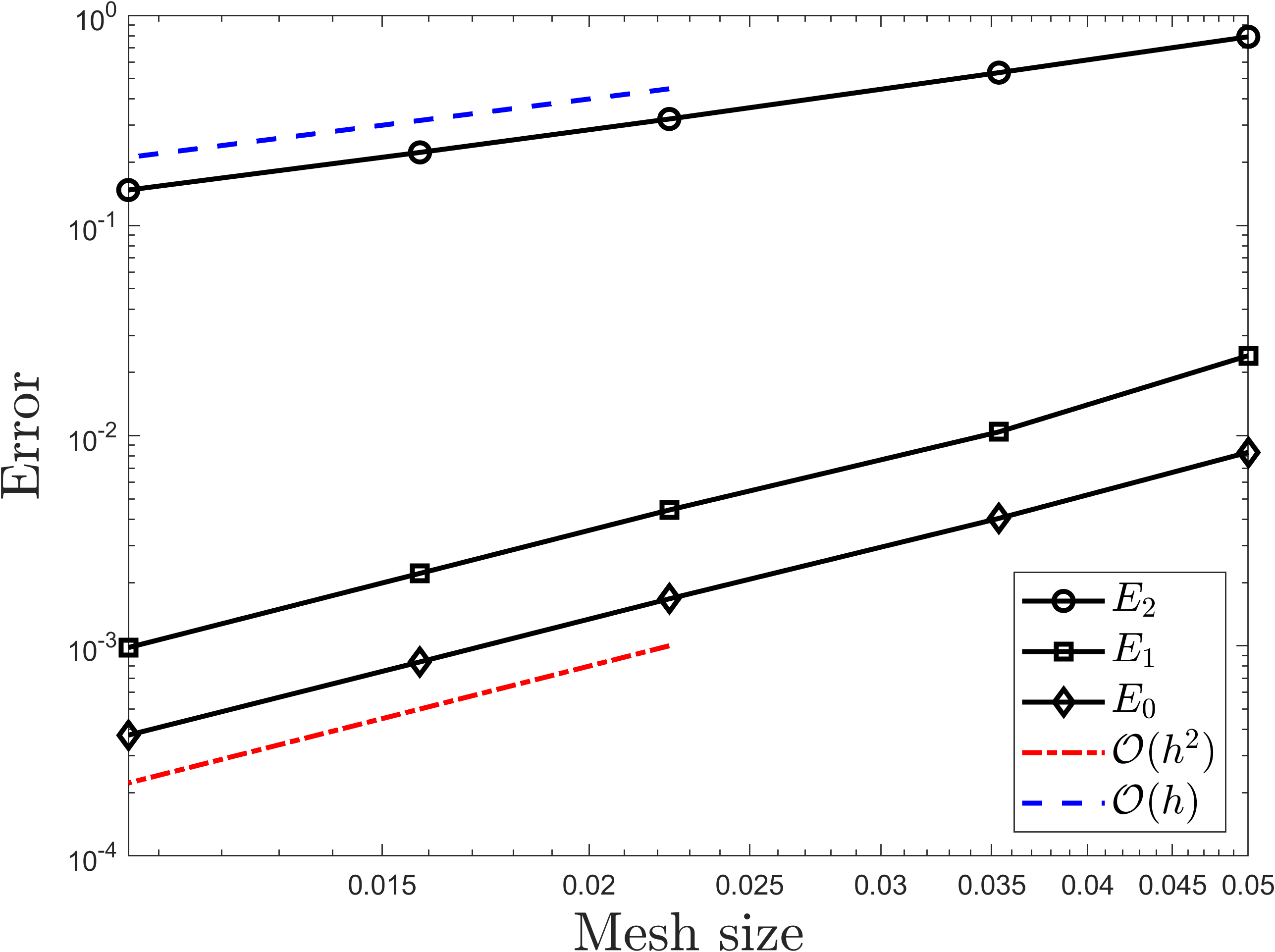}
    }
\caption{Errors for $C^1$-conforming (left) and $C^0$-nonconforming (right) VEM on the mesh $\mathcal K^4$}
    \label{fig:ratemeshdata}
\end{figure}

\begin{figure}
    \centering
    \subfigure{
        \includegraphics[width=0.47\textwidth]{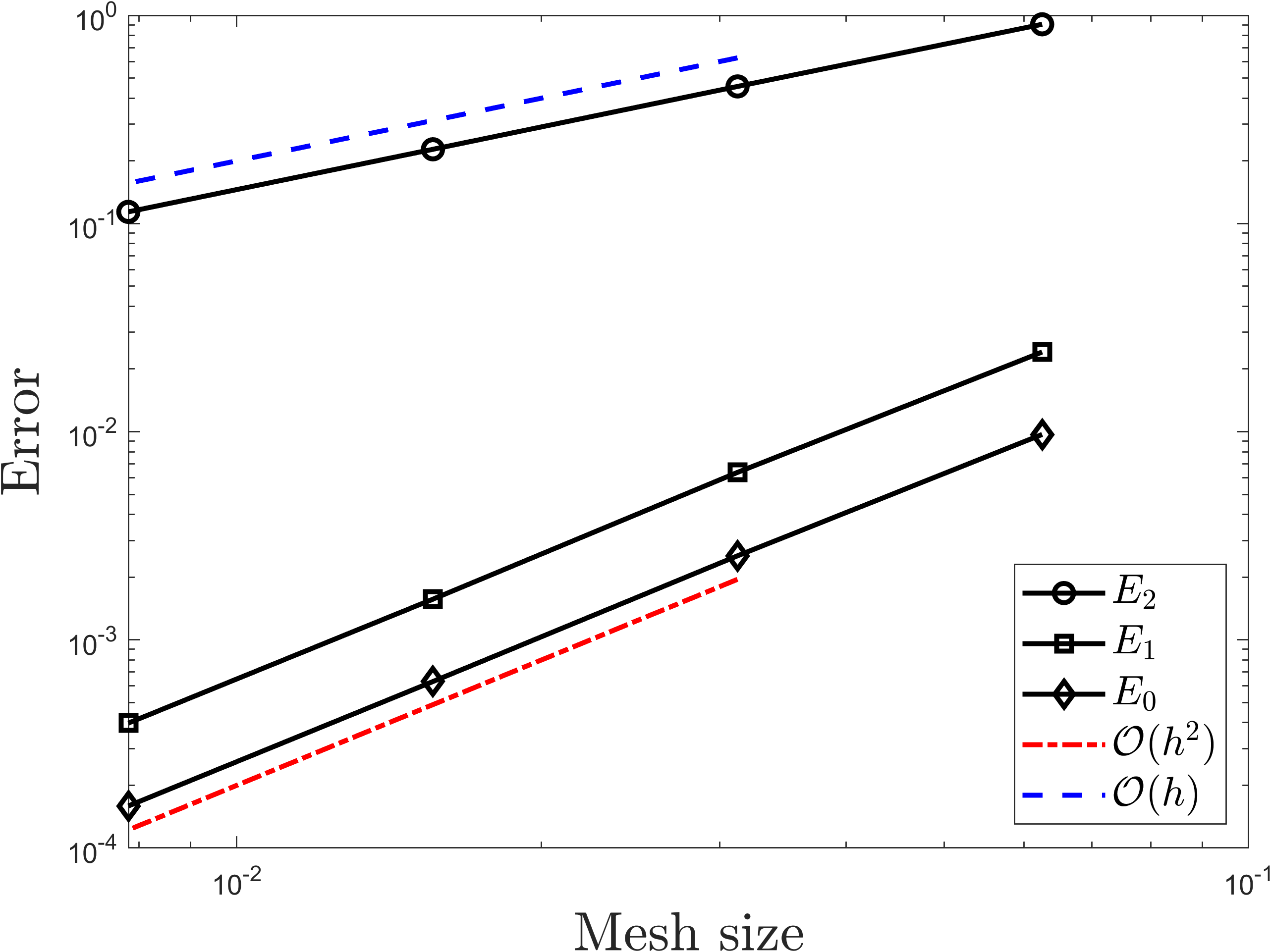}
    }
    \subfigure{
	\includegraphics[width=0.47\textwidth]{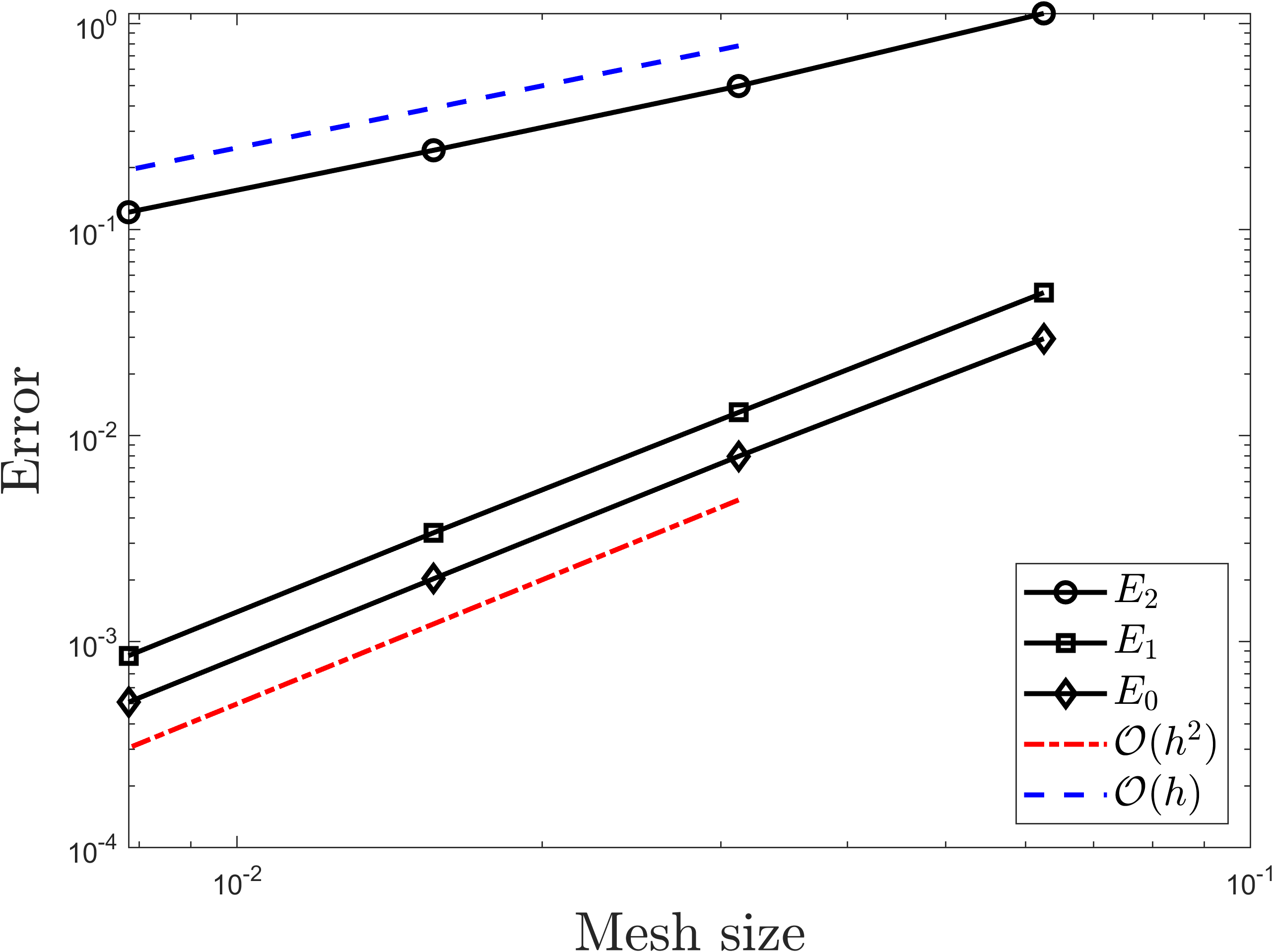}
    }
\caption{Errors for $C^1$-conforming (left) and $C^0$-nonconforming (right) VEM on the mesh $\mathcal K^5$}
    \label{fig:rateperturbRect}
\end{figure}

\subsection{Example 2} In this test, we consider an example from \cite{SS2014,Wu2021}. We solve the nonlinear HJB equation on the domain $\varOmega=(0,1)^2$.
The compact metric space is taken as $\varLambda=[0,\pi/3]\times\mathrm{SO}(2)$, where $\mathrm{SO}(2)$ represents the set
consisting all $2\times2$ rotation matrices. The coefficients matrix $\A:=\frac12\sigma^\alpha(\sigma^\alpha)^T$ with
\[\sigma^\alpha= R\begin{pmatrix}
                           1 & \sin\theta \\
                           0 & \cos\theta
                         \end{pmatrix}
                         \text{ and }
                         \alpha=(\theta,R)\in \varLambda.
                         \]
Additionally, $\bm b^\alpha=\bm0$ and $c^\alpha=\pi^2$. By choosing $\lambda=8\pi^2/7$, we can verify that the Cordes condition holds  with $\varepsilon=1/7$ though $\tr\A^\alpha=1$ and $|\A^\alpha|=(1+\sin^2\theta)/2\leq7/8$.
The exact solution is chosen as \[u=\mathrm e^{x_1x_2}\sin(\pi x_1)\sin(\pi x_2),\] and the load term $f^\alpha=\sqrt3\sin^2\theta/\pi^2+g$, where $g$ is a function determined by the exact solution as well as the governing equation, and it is independent of $\alpha$. We employ the semismooth Newton’s method introduced in Section \ref{sec:SemismoothNewton} to solve this problem, and utilize the triangle mesh $\mathcal K^1$ as the underlying mesh.
The iteration will stop when \[\left(\sum_{K\in\mathcal K_h}\|\Pi_0^KD^2u_h^{j+1}-\Pi_0^KD^2u_h^j\|_{0,K}^2\right)^{\frac12}<10^{-8}.\]
We denote
\[E_4:=\left(\sum_{K\in\mathcal K_h}\|D^2u-\Pi_0^KD^2u_h^j\|_{0,K}^2\right)^{\frac12}.\]
We display the convergence properties and the iteration number of the semismooth Newton’s method for $C^1$-conforming VEM in Figure \ref{fig:Ex2com}.
 As predicted by Theorem \ref{thm:errC}, $D^2u_h$ converges to $D^2u$ with an order of $\mathcal O(h)$, whereas $u_h$ and $\nabla u_h$ converge at the rate of $\mathcal O(h^2)$.
For the numerical results of the $C^0$-nonconforming VEM, which are presented is Figure \ref{fig:Ex2noncom}, a similar convergence property can be observed.

\begin{figure}
    \centering
    \subfigure{
        \includegraphics[width=0.47\textwidth]{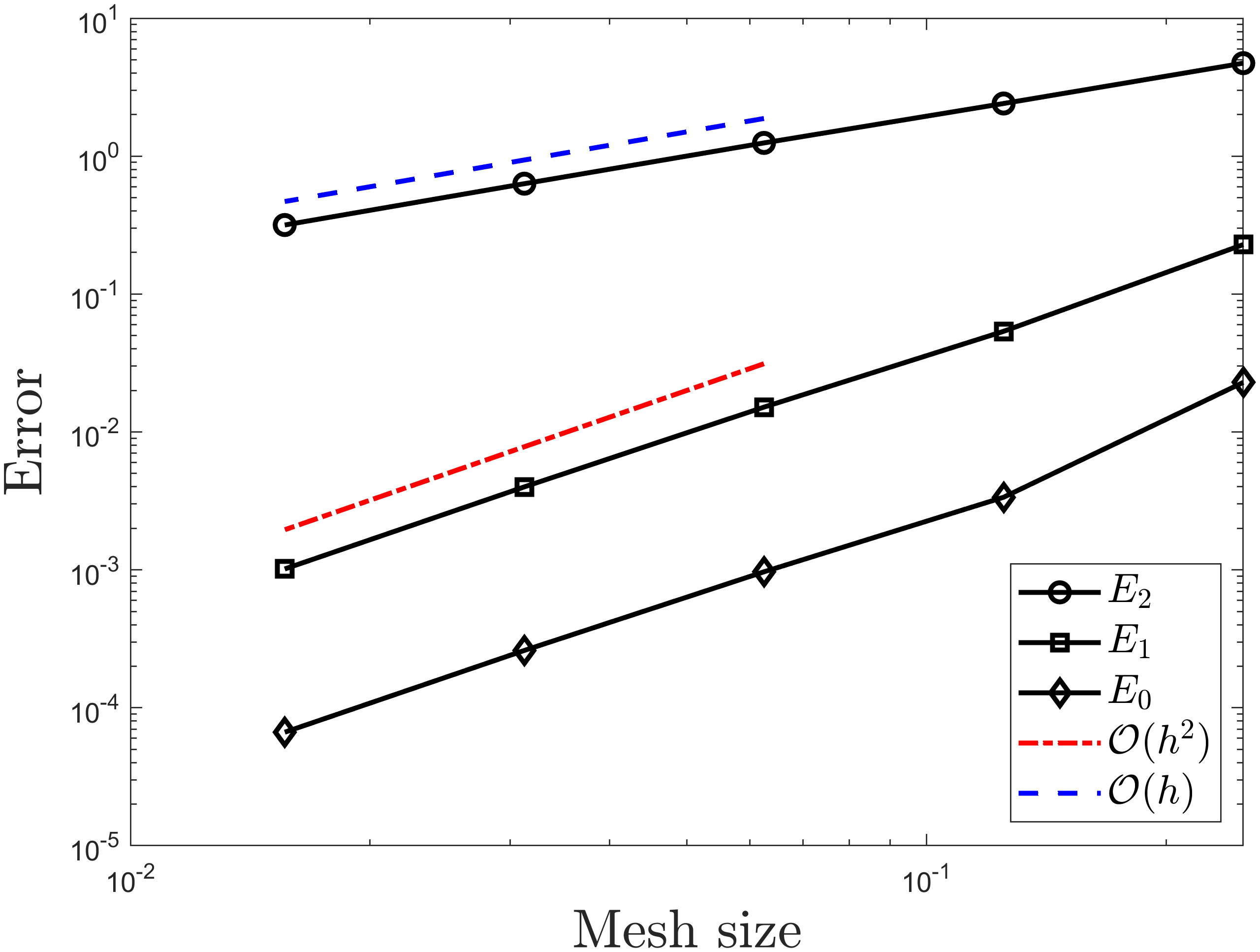}
    }
    \subfigure{
	\includegraphics[width=0.47\textwidth]{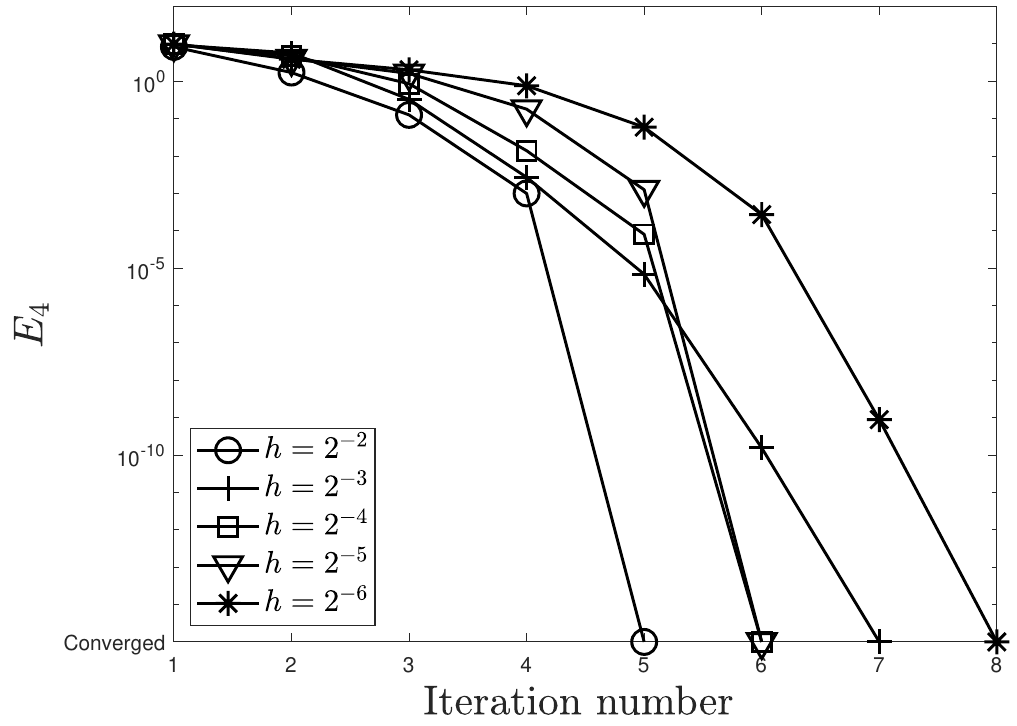}
    }
\caption{Convergence rate and iteration number for $C^1$-conforming  VEM on the mesh $\mathcal K^1$}
    \label{fig:Ex2com}
\end{figure}

\begin{figure}
    \centering
    \subfigure{
        \includegraphics[width=0.47\textwidth]{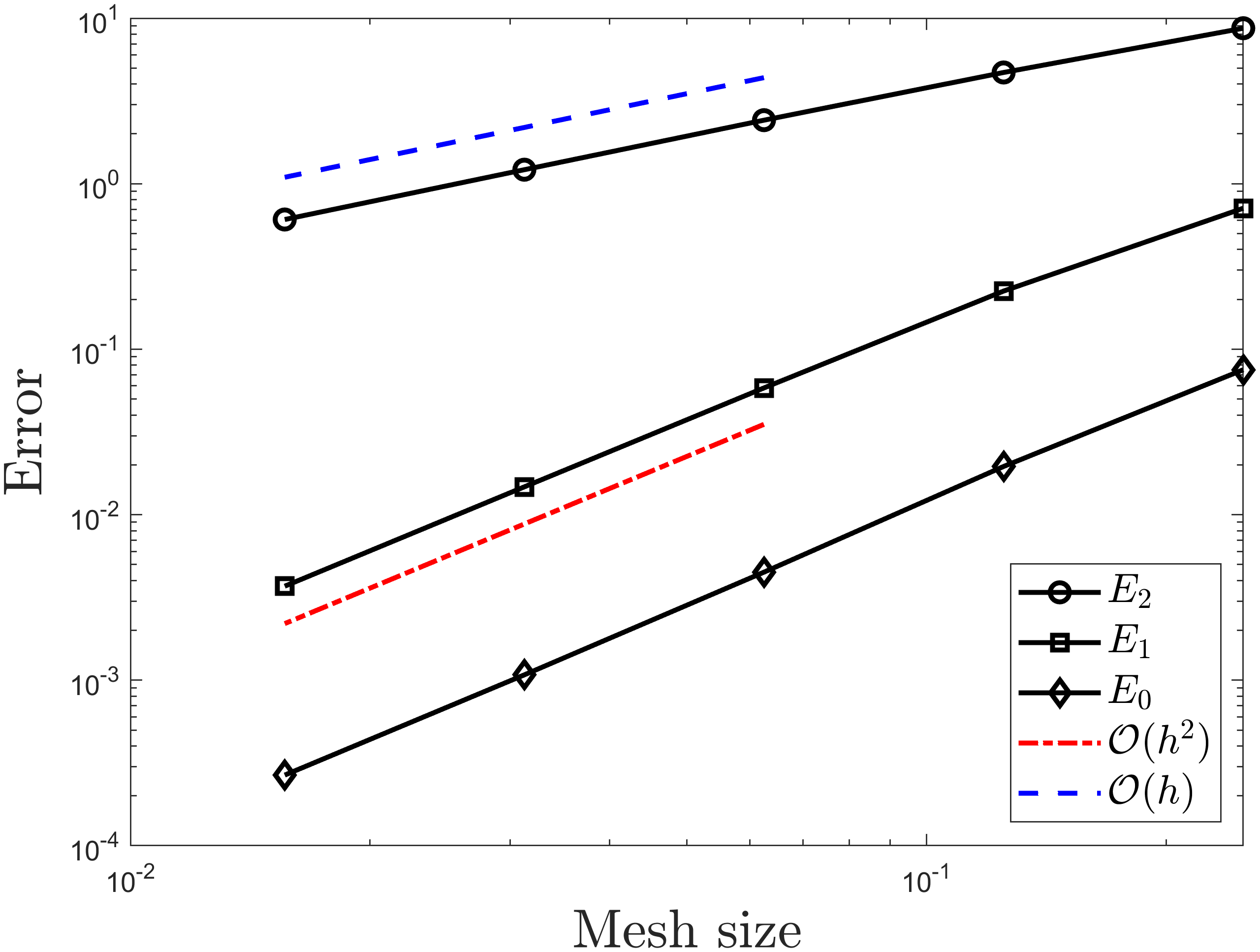}
    }
    \subfigure{
	\includegraphics[width=0.47\textwidth]{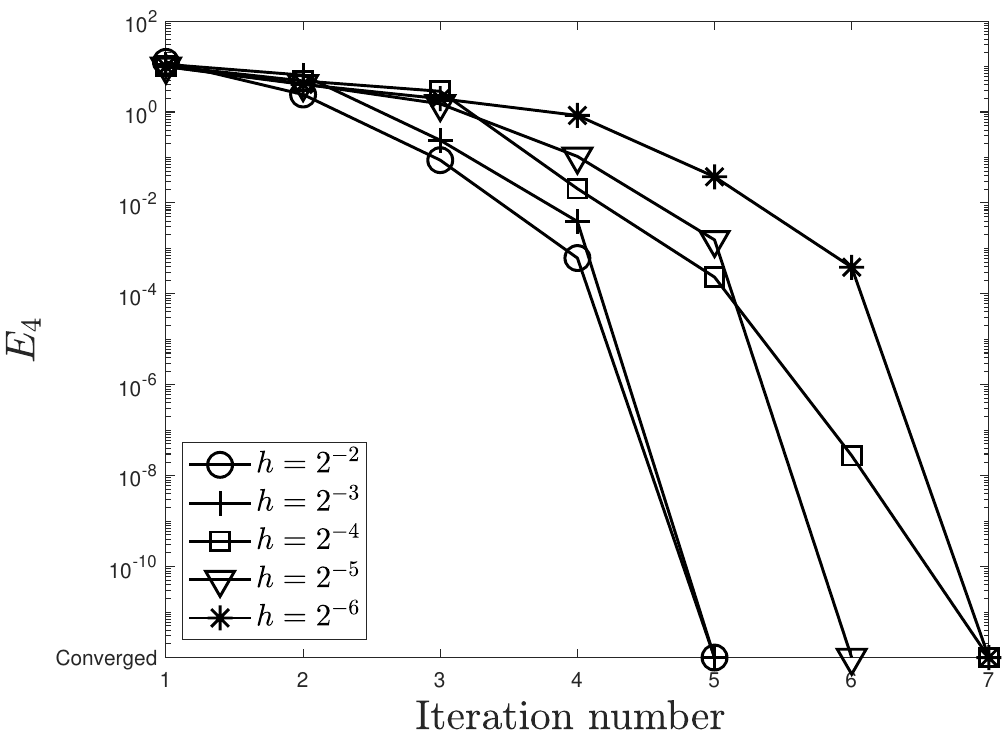}
    }
\caption{Convergence rate and iteration number for $C^0$-nonconforming  VEM on the mesh $\mathcal K^1$}
    \label{fig:Ex2noncom}
\end{figure}

\subsection{Example 3} Referring to Test 3 in \cite{NW2019}, we consider a similar example. The computational domain is the square $\varOmega=\{(x_1,x_2): |x_1|<\pi, |x_2|<\pi\}$. The compact metric set  $\varLambda$ is $\{1,2\}$, and the coefficients are defined as
\begin{align*}
 & \A^1=\begin{pmatrix}
         2 & 1/2 \\
         1/2 & 3/2
       \end{pmatrix}+\frac{x_1}{|x_1|}\frac{x_2}{|x_2|} \begin{pmatrix}
         1 & 1/2 \\
         1/2 & 1/2
       \end{pmatrix},  \\
       &
        \A^2=\begin{pmatrix}
         3/2 & 1/2 \\
         1/2 & 2
       \end{pmatrix}+\frac{x_1}{|x_1|}\frac{x_2}{|x_2|} \begin{pmatrix}
         1/2 & 1/2 \\
         1/2 & 1
       \end{pmatrix},\\
   & \bm b^1=\bm b^2=(1~~0)^T,\qquad c^1=c^2=1.
\end{align*}
The exact solution is  $u=\sin x_1\sin x_2$,  and the source function set $\{f^\alpha\}_{\alpha\in\varLambda}$ is determined by the exact solution. By taking $\lambda=1$ we see that the Cordes condition holds with $\varepsilon=1/6$. We use the same stopping criteria as in Example 2 for the semismooth Newton’s method and display the numerical results in Figure
\ref{fig:Ex3}. 
Despite the fact that the coefficient $\A^\alpha$, $\alpha\in\varLambda$, is discontinuous with respect to $\alpha$, and that $\A^1$
  and $\A^2$  exhibit discontinuities at $x=0$ or $y=0$ within this example, 
we still observe that $D^2u_h$ converges to $D^2u$ with an order of  $\mathcal O(h)$, as predicted by Theorem \ref{thm:errC}. Additionally, the convergence rate of $\mathcal O(h)$ is also observed for errors $E_1$ and $E_0$.

\begin{figure}
    \centering
    \subfigure{
        \includegraphics[width=0.47\textwidth]{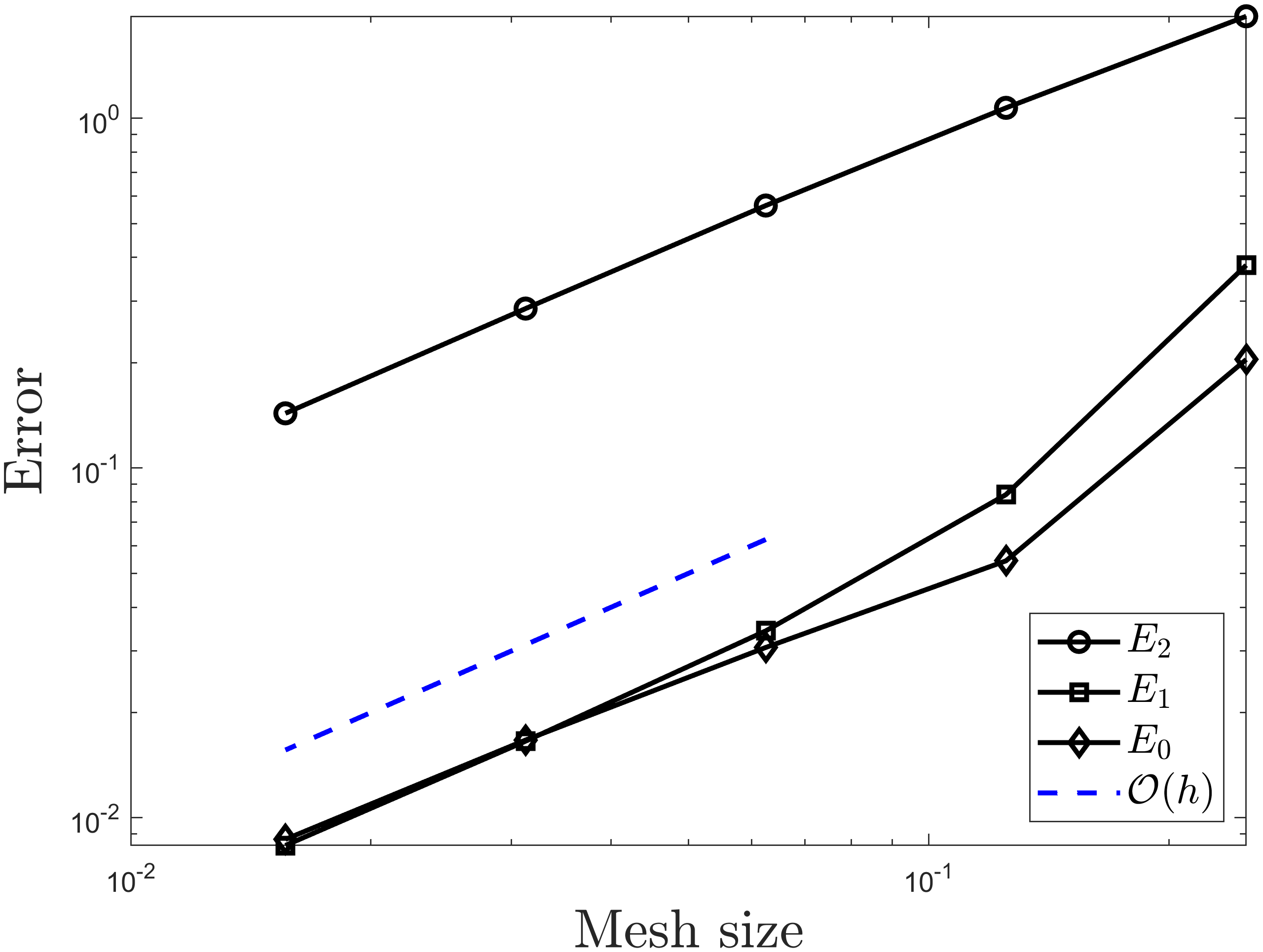}
    }
    \subfigure{
	\includegraphics[width=0.47\textwidth]{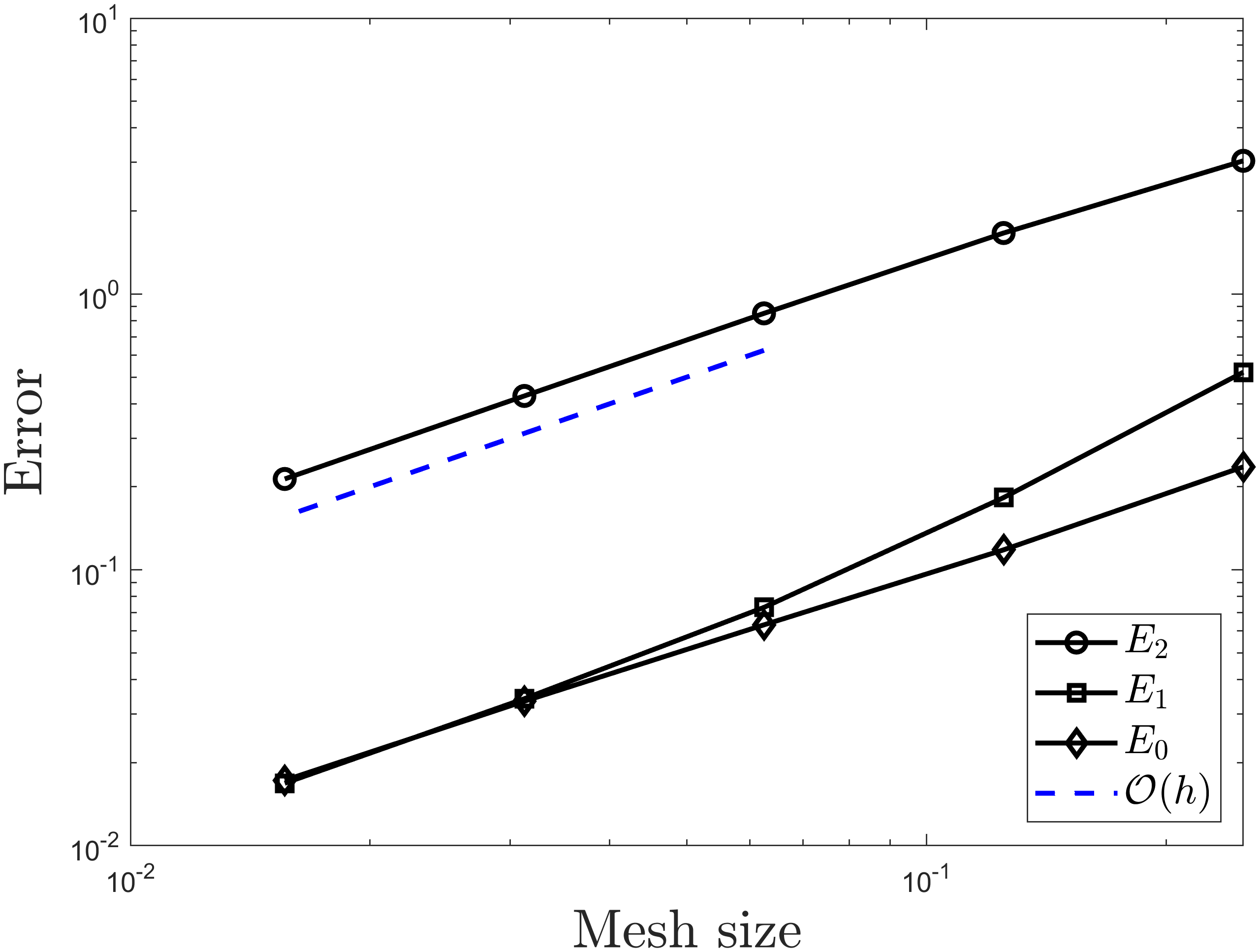}
    }
\caption{Errors for $C^1$-conforming (left) and $C^0$-nonconforming (right) VEM on the mesh $\mathcal K^1$}
\label{fig:Ex3}
\end{figure}

\section{Conclusion}\label{sec:conclusion}
We have developed and rigorously analyzed both $C^1$-conforming and $C^0$-nonconforming virtual element methods for Hamilton-Jacobi-Bellman equations. Detailed \textit{a priori} error estimates for our formulations have been derived. Furthermore, the semismooth Newton's method has been proposed for the numerical solution of the discrete systems. The methodologies presented herein demonstrate potential applicability to other nonlinear problems, including fully nonlinear Monge–Ampère equations and Isaacs equations, which will be systematically explored in future research.

\section*{Acknowledgment}
This work was supported in part by the Andrew Sisson Fund, Dyason Fellowship, the Faculty Science Researcher Development Grant of the University of Melbourne, and the NSFC grant 12131005.

\bibliographystyle{siamplain}
\bibliography{mybibfile}

\end{document}